\documentclass[reqno,11pt,a4paper]{amsart}

\usepackage{exscale,amsmath,amssymb,amsthm}
\usepackage[margin=1in]{geometry}

\begin{document}
\newcommand{\stablebetaminusonewith}{\alpha}
\newcommand{\stablebetaminusonewithout}{\alpha}
\newcommand{\stablebetaminusonespecial}{\alpha}
\newcommand{\stablebetawith}{(1+\alpha)}
\newcommand{\stablebetawithout}{{1+\alpha}}
\newcommand{\stablebetaindex}{\alpha}
\newcommand{\stablebetaspecial}{(1+\alpha)}

\newcommand{\localtimeholder}{\nu}
\newcommand{\pathholder}{\gamma}
\newcommand{\diversityalpha}{\beta}

\newcommand{\thresholda}{b}

\newcommand{\markZ}{Z}

\newcommand{\discretisationn}{k}

\newcommand{\bC}{\mathbb{C}}
\newcommand{\bD}{\mathbb{D}}
\newcommand{\bE}{\mathbb{E}}
\newcommand{\bI}{\mathbb{I}}
\newcommand{\bM}{\mathbb{M}}
\newcommand{\bN}{\mathbb{N}}
\newcommand{\bP}{\mathbb{P}}
\newcommand{\bQ}{\mathbb{Q}}
\newcommand{\bR}{\mathbb{R}}
\newcommand{\BR}{\mathbb{R}}
\newcommand{\bT}{\mathbb{T}}
\newcommand{\bU}{\mathbb{U}}
\newcommand{\bZ}{\mathbb{Z}}

\newcommand{\cB}{\mathcal{B}}
\newcommand{\cC}{\mathcal{C}}
\newcommand{\cL}{\mathcal{L}}
\newcommand{\cM}{\mathcal{M}}
\newcommand{\cN}{\mathcal{N}}
\newcommand{\cP}{\mathcal{P}}
\newcommand{\cT}{\mathcal{T}}
\newcommand{\cR}{\mathcal{R}}

\newcommand{\Exc}{\mathcal{E}}
\newcommand{\SExc}{\Sigma(\Exc)}
\newcommand{\cadlag}{\textrm{c\`adl\`ag}}
\def\ShiftRestrict#1#2{#1\big|^{\from}_{#2}} 
\def\shiftrestrict#1#2{#1|^{\from}_{#2}}
\def\RestrictShift#1#2{\prescript{\to}{#2}{\big|}#1} 
\def\restrictshift#1#2{\prescript{\to}{#2}|#1}
\def\Restrict#1#2{#1\big|_{#2}}
\def\restrict#1#2{#1|_{#2}}
\def\life{\zeta}

\newcommand{\cf}{\mathbf{1}}

\newcommand{\eps}{{\ensuremath{\varepsilon}}}
\newcommand{\besq}{\mathrm{BESQ}}

\newcommand{\ft}{\mathbf{t}}
\newcommand{\fs}{\mathbf{s}}

\newtheorem{lemma}{Lemma}
\newtheorem{proposition}[lemma]{Proposition}
\newtheorem{corollary}[lemma]{Corollary}
\newtheorem{theorem}[lemma]{Theorem}
\newtheorem{problem}{Problem}
\newtheorem{problem17a}{Problem 17a}
\newtheorem{aside}{Aside}
\newtheorem{claim}{Claim}

\newenvironment{pfofprop14}{\begin{trivlist}\item[] \textbf{Proof of Proposition \ref{prop14}.}}
                     {\hspace*{\fill} $\square$\end{trivlist}}

\newenvironment{pfofthmltp}{\begin{trivlist}\item[] \textbf{Proof of Theorem \ref{thmltp}.}}
                     {\hspace*{\fill} $\square$\end{trivlist}}

\newcommand{\ed}{\mbox{$ \ \stackrel{d}{=}$ }}
\newcommand{\giv}{\,|\,}
\newcommand{\convd}{\overset{d}{\underset{{n\rightarrow \infty}}{\longrightarrow}}}
\newcommand{\convdk}{\overset{d}{\underset{{k\rightarrow \infty}}{\longrightarrow}}}
\newcommand{\convask}{\overset{a.s.}{\underset{{k\rightarrow \infty}}{\longrightarrow}}}

\newcommand{\convv}{\overset{v}{\underset{{n\rightarrow \infty}}{\longrightarrow}}}
\newcommand{\convvz}{\overset{v}{\underset{{z\downarrow 0}}{\longrightarrow}}}
\newcommand{\convdz}{\overset{d}{\underset{{z\downarrow 0}}{\longrightarrow}}}

\title[Uniform control of local times of spectrally positive stable processes]{Uniform control of local times\\ of spectrally positive stable processes}

\author{N\MakeLowercase{\sc oah} F\MakeLowercase{\sc orman,} S\MakeLowercase{\sc oumik} P\MakeLowercase{\sc al,} D\MakeLowercase{\sc ouglas} R\MakeLowercase{\sc izzolo, and}  M\MakeLowercase{\sc atthias} W\MakeLowercase{\sc inkel}}

\address{\hspace{-0.42cm}Noah~Forman\\ Department of Statistics\\ University of Oxford\\ 24--29 St Giles'\\ Oxford OX1 3LB, UK\\ Email: noah.forman@gmail.com}             

\address{\hspace{-0.42cm}Soumik~Pal\\ Department of Mathematics \\ University of Washington\\ Seattle WA 98195, USA\\ {Email: soumikpal@gmail.com}}

\address{\hspace{-0.42cm}Douglas~Rizzolo\\ Department of Mathematical Sciences\\ University of Delaware\\ Newark DE 19716, USA\\ Email: drizzolo@udel.edu}

\address{\hspace{-0.42cm}Matthias~Winkel\\ Department of Statistics\\ University of Oxford\\ 24--29 St Giles'\\ Oxford OX1 3LB, UK\\ Email: winkel@stats.ox.ac.uk}

\keywords{Stable process, squared Bessel processes, CMJ process, local time approximation, excursion theory, restricted L\'evy process, H\"older continuity}

\subjclass[2010]{60G51, 60G52, 60J55, 60J80}

\date{\today}

\thanks{This research is partially supported by NSF grants DMS-1204840, DMS-1308340, DMS-1612483, UW-RRF
grant A112251, and EPSRC grant EP/K029797/1}

\begin{abstract} We establish two results about local times of spectrally positive stable processes. The first is a general approximation result, uniform in 
  space and on compact time intervals, in a model where each jump of the stable process may be marked by a random path. The second gives moment control on the H\"older
  constant of the local times, uniformly across a compact spatial interval and in certain random time intervals. For the latter, we introduce the notion of a 
  L\'evy process restricted to a compact interval, which is a variation of Lambert's L\'evy process confined in a finite interval and of Pistorius' doubly reflected
  process. We use the results of this paper to exhibit a class of path-continuous branching processes of Crump-Mode-Jagers type with continuum genealogical structure.
  A further motivation for this study lies in the construction of diffusion processes in spaces of interval partitions and $\bR$-trees, which we explore in 
  forthcoming articles. 
\end{abstract}

\maketitle

\vspace{-0.5cm}

\section{Introduction}

\noindent Consider a spectrally positive L\'evy process $X=(X_t,t\ge 0)$ with Laplace exponent $\psi(\eta)=\eta^\stablebetawithout $, for some $\stablebetaminusonespecial\in(0,1)$. We mark each jump $(t,\Delta X_t)$ 
independently by a continuous random nonnegative path $\markZ_t=(\markZ_t(s),0\le s\le\Delta X_t)$ with $\markZ_t(0)=\markZ_t(\Delta X_t)=0$, using a marking kernel $x\mapsto\kappa_q(x,\cdot)$ that 
has the following self-similarity property for some $q>0$. If $\markZ$ has distribution $\kappa_q(1,\cdot)$, then $(x^{q}\markZ(s/x),0\le s\le x)$ has distribution 
$\kappa_q(x,\cdot)$. We consider the levels in the jump interval $[X_{t-},X_t]$ as the (vertical!) time interval of the path, associating $\markZ_t(s)$ with level $X_{t-}+s$. 
If we interpret each $\markZ_t(s)$ as a population size or \em mass \em at level $X_{t-}+s$, the aggregate at time $y$ associated with $(X_t,0\le t\le T)$ is given by
\begin{equation}\label{cmj}\markZ_{[0,T]}(y)=\sum_{0\le t\le T}\markZ_t(y-X_{t-}),\qquad y\in\mathbb{R},\ T\ge 0,
\end{equation}
with the convention that $\markZ_t(s)=0$ for $s\not\in(0,\Delta X_t)$. Structurally, (\ref{cmj}) is the same as the sum of characteristics in a 
Crump-Mode-Jagers process, see e.g. \cite{Jagers1974,Jagers1975}. Connections between 
the genealogy of certain CMJ processes (without characteristics) and L\'evy processes were established by Lambert and Uribe Bravo 
\cite{Lambert,LambertUribe2016arxiv}. L\'evy processes marked by Poisson processes were studied by Delaporte \cite{DelaporteII,Delaporte15}. Here, we take some of
those ideas further by allowing characteristics $\markZ_t$ that vary during the existence interval $[X_{t-},X_t]$ of the population indexed by $t\in[0,T]$. We focus on the 
local times of $X$ and their interpretation as the genetic diversity of the population. 

Boylan \cite{Boy1964} showed that $X$ possesses jointly continuous local times $(\ell^y(t),y\in\mathbb{R},t\ge 0)$. Barlow, Perkins and Taylor \cite{BPT1,BPT2} established
several uniform approximations of local times. Our first goal in this paper is to give a new intrinsic approximation of local times $\ell^y(t)$ using the marks at level $y$. We define for all $h>0$, $y\in\mathbb{R}$ and $T\ge 0$
\begin{equation}\label{mhyt} m_h(y,T)=\sum_{0\le t\le T}1{\{\markZ_t(y-X_{t-})\ge h\}}=\#\{0\le t\le T\colon \markZ_t(y-X_{t-})\ge h\}.
\end{equation}  
This sum can be seen as a sum over excursions of $X$ away from level $y$. Since $X$ is spectrally positive, each such excursion has (at most) one jump across level $y$,
so each excursion contributes just one term to the sum. 

\begin{theorem}\label{thm1} Let $\stablebetaminusonespecial  \in(0,1)$ and $q>\alpha$. Suppose that $\markZ\sim\kappa_q(1,\cdot)$ is $\pathholder$-H\"older continuous 
  for some $\pathholder\in(0,q)$, that the H\"older 
  constant\vspace{-0.2cm}
  $$D_\pathholder:=\sup_{0\le r<s\le 1}\frac{|\markZ(s)-\markZ(r)|}{|s-r|^\pathholder}$$
  has moments of all orders, and that $c:=\stablebetawith (\Gamma(1-\stablebetaminusonespecial))^{-1}\bE((\markZ(U))^{\stablebetaminusonewith /q})>0$, where $U$ is a uniform variable on $[0,1]$ that is  
  independent of $\markZ$. Then the following holds almost
  surely for a $\kappa_q$-marked spectrally positive stable L\'evy process with Laplace exponent $\psi(\eta)=\eta^{\stablebetawithout}$: 
  $$\lim_{h\downarrow 0}\,\sup_{0\le t\le T}\,\sup_{y\in\mathbb{R}}\left|\frac{h^{\stablebetaminusonewith /q}}{c}m_h(y,t)-\ell^y(t)\right|=0\qquad\mbox{for all $T>0$.}$$ 
\end{theorem}

To prove this result, we use excursion theory. Some steps resemble arguments of
Khoshnevisan \cite{Kho-94}, who established rates of convergence of certain uniform approximations of the local times of Brownian motion. 

The case where the paths are squared Bessel bridges/excursions of dimension $\delta\in\bR$ seems most compelling, in view of the branching structure (with immigration or
emigration) of these processes that make them natural models for population size evolution, and in view of the Ray-Knight theorems. Specifically, the cases $\delta=0$ and 
$\delta=2$ appear in Ray-Knight theorems as distributions of Brownian local time as a process in the spatial variable. Generalisations to perturbed Brownian
motions include other parameters, see e.g. \cite{Perman1996,RY}. Squared Bessel processes and their excursions and bridges have self-similarity parameter $q=1$.  
H\"older continuity of these processes is well-known, and we show in Corollary \ref{cor general bridge holder} that the moment assumption of Theorem 
\ref{thm1} also holds for $\pathholder\in(0,1/2)$. We restate the special case of Theorem \ref{thm1} as Theorem \ref{thm1besq}.

Other interesting paths include some that are deterministic functions of the jump height, so that marking does not introduce additional randomness. The following example 
is obtained for deterministic $z(s)=s\wedge(1-s)$, $0\le s\le 1$, and $\kappa(1,\cdot)=\delta_z$, using $q=1$ to scale $z$ to $z_x(s)=s\wedge(x-s)$, $0\le s\le x$, and $\kappa(x,\cdot)=\delta_{z_x}$:

\begin{corollary}\label{corhat} Consider a spectrally positive stable process of index $\stablebetawithout \in(1,2)$ and let 
$$m_h^\circ(y,T)=\#\{0\le t\le T\colon X_{t-}+h\le y\le X_t-h\},\qquad h>0,\ y\in\mathbb{R},\ T\ge 0.$$
Let $c^\circ=2^\alpha\Gamma(1-\alpha)$. Then the following holds almost
  surely for all $T>0$
  $$\lim_{h\downarrow 0}\,\sup_{0\le t\le T}\,\sup_{y\in\mathbb{R}}\left|c^\circ h^{\stablebetaminusonewithout }m_h^\circ(y,t)-\ell^y(t)\right|=0.$$ 
\end{corollary}

This is clearly related to approximations by counting upcrossings \cite{CLJPT,Kho-94} or by jump heights. However, while we could obtain the latter by setting 
$z(s)=1$, this violation of the continuity requirement on $z$ seriously affects our arguments. The former cannot be expressed in this form, since whether an excursion 
contains an upcrossing depends not only on the single jump across the level, but on the pre- or post-jump parts of the excursion, as an upcrossing may be achieved by 
multiple upward jumps.

Let us turn to our second main result. We denote by $\tau^y(s)=\inf\{t\ge 0\colon\ell^y(t)>s\}$ the inverse local time of $X$ in level $y\in\bR$. It is a well-known 
consequence of the scaling property of $X$ that $\tau^0$ is a stable subordinator and hence has infinite mean. To obtain the following result, we restrict $X$ to space $[0,\thresholda]$ by making jumps across $\thresholda$ stop short at $\thresholda$, by
excluding excursions above $\thresholda$ and below 0, and similarly treating jumps from below 0. Similar processes have been studied by Lambert \cite{Lambert2000}, Pistorius 
\cite{Pis03}, and indeed very recently Lambert and Uribe Bravo \cite{LambertUribe2016arxiv} proceeded similarly by restricting just at the upper boundary. See Section \ref{sectrestricted} for 
a precise discussion. We can state the theorem without this notion of a restricted L\'evy process.  

\begin{theorem}\label{thmltholdermoment} Let $\stablebetaminusonespecial\in(0,1)$, $\thresholda>0$, $\localtimeholder\in(0,\stablebetaminusonewith /2)$ and $S$ a random variable
  with moments of all orders. Then the local times of a spectrally positive stable process of index $\stablebetawithout$ have a uniform H\"older constant\vspace{-0.1cm}
  $$D_\localtimeholder^{[0,\thresholda]}(\tau^0(S)):=\sup_{0\le t\le\tau^0(S)}\,\sup_{0\le x<y\le \thresholda}\frac{|\ell^y(t)-\ell^x(t)|}{|y-x|^\localtimeholder}$$
  with moments of all orders.
\end{theorem}

Uniform H\"older continuity of local times was already obtained by Boylan \cite{Boy1964}. The novelty of our result is in the finiteness of moments of the H\"older 
constant, which appears to be a fundamental fact about L\'evy processes. The finiteness of H\"older constants is useful when dealing with families of scaled independent 
copies. To illustrate this point (in the context of Theorem \ref{thm1}, see Proposition \ref{propuniformholder}) and other aspects of the methods used in this paper, we devote 
the remainder of this introduction to the exploration of some auxiliary results, which may also 
be of independent interest. In particular, we will apply some of the auxiliary results to study some properties of the process $(y,t)\mapsto \markZ_{[0,t]}(y)$ of 
(\ref{cmj}).

Consider the It\^o excursion measure $n$ of $X$ away from $0$. Denote by $A$ and $B$ the undershoot and overshoot of the unique jump across level 0 under $n$, by
$H=A+B$ the jump size and by $U=A/H$ the relative undershoot. The classical study of excursions away from $0$ was recently complemented by Pardo et al. \cite{PPR15arxiv}.
In the stable case, we can take their study further and establish the following result. Symmetry under time-reversal is due to Getoor and Sharpe \cite{GetoShar81}.

\begin{proposition}\label{propunif} We have $n(A\in dy,B\in dz)=\stablebetawith\stablebetaminusonewith(\Gamma(1-\stablebetaminusonespecial))^{-1}(y+z)^{-\stablebetaminusonespecial-2}dzdy$ and
  $n(H\in dx,U\in du)=\stablebetawith\stablebetaminusonewith(\Gamma(1-\stablebetaminusonespecial))^{-1}x^{-\stablebetaminusonespecial-1}dxdu$. In particular, under $n(\,\cdot\,|\,H=x)$, 
  $U$ is uniform on $[0,1]$. Under $n(\,\cdot\,|\,A=y,B=z)$, the pre- and post-jump parts of the excursion are independent. The post-jump
  part is a stable process starting from $z$ until hitting 0. The pre-jump part is the negative of the time-reversal of a stable process starting from $y$
  until hitting~0.
\end{proposition}

Recall that Theorem \ref{thm1} assumes the existence of moments for the H\"older constants of the random paths under $\kappa_q(1,\cdot)$. The $\kappa_q$-marked stable
process contains a countably infinite collection of scaled independent copies of this random path. We can use the moments on the H\"older constants assumed in Theorem \ref{thm1} in a simple Poisson
sum over jumps to obtain uniform H\"older constants (not with all moments since the stable L\'evy measure has a heavy tail). 

\begin{proposition}\label{propuniformholder} In the setting of Theorem \ref{thm1}, the uniform H\"older constant for all paths $\markZ_t$ is finite a.s.: for all $T>0$, 
  $$D_{[0,T]}:=\sup_{0\le t\le T}\,\sup_{0\le x<y\le\Delta X_t}\frac{|\markZ_t(y)-\markZ_t(x)|}{|y-x|^\pathholder}<\infty.$$
\end{proposition}

This is remarkable, since the unbounded variation of $X$ means that the jump heights are not summable, hence the time intervals of the processes $\markZ_t$, $0\le t\le T$, add
up to infinite length. We can further improve on this proposition, as follows, to facilitate the aggregation (\ref{cmj}) of mass at level $y$, for all levels
$y\in\mathbb{R}$.

\begin{proposition}\label{proppiling} In the setting of Theorem \ref{thm1}, let $\gamma<q-\alpha$. Then the set $\{t\in[0,T]\colon\Delta X_t>0\}$ of jump times may 
  a.s.\ be partitioned into sequences 
  $\{J^k_j,\ j\ge 1\}$, $k\ge 1$, in such a way that 
 \begin{enumerate}
  \item the jump intervals $[X(J^k_j-),X(J^k_j)]$, $j\ge 1$, are disjoint for each $k\ge 1$;
  \item H\"older constants $\displaystyle D_{[0,T]}^{k}=\sup_{j\ge 1}\sup_{0\le x<y\le\Delta X(J^k_j)}\frac{|\markZ_{J^k_j}(y)-\markZ_{J^k_j}(x)|}{|y-x|^\pathholder}$ 
    are summable in $k\ge 1$.
 \end{enumerate} 
\end{proposition}

Let us return to the aggregate process (\ref{cmj}) of Crump-Mode-Jagers type. We begin by noting some properties that illustrate the genealogical complexity. Jagers' 
\cite{Jagers1974} framework is very general, but only covers branching processes with a discrete genealogy that can be represented as a subset of the infinite
Ulam-Harris tree $\mathbb{U}=\bigcup_{n\ge 0}\bN^n$. Lambert \cite{Lambert} used this framework to encode a class of CMJ processes in bounded variation L\'evy processes with negative drift, in such a way that the jumps correspond to the members of $\mathbb{U}$. Since all jumps happen at
local minima, there is a unique previous jump across that local minimum level that corresponds to the parent. From any given jump, this yields a finite number of ancestors down to 
level 0. 

In our framework, $X$ has unbounded variation, no jumps occur from local minima (see e.g. \cite[Chapter VII]{Ber1996}), hence no direct parent for any jump can be
identified in this way, but a collection of ancestral intervals and a continuum of limit levels between these can be put together to a continuum genealogy. This is very 
delicate and properly addressed in \cite{LambertUribe2016arxiv}. By marking the jumps of the stable process, we have by-passed the subtleties of the continuum genealogy 
and set up $\markZ_{[0,T]}$ directly building on the jump structure of the stable process, which by \cite{LambertUribe2016arxiv} allows the CMJ type interpretation. 

The results and methods we have discussed yield some properties of $\markZ_{[0,T]}$, while an exhaustive study of this interesting new class of processes is beyond 
the scope of the present paper. 

\begin{corollary}\label{corcmj} In the setting of Theorem \ref{thm1}, the process $y\mapsto \markZ_{[0,T]}(y)$ of \eqref{cmj} is a.s. $\pathholder$-H\"older.  
\end{corollary}

Recall that excursions of $X$ away from level $y$ have (at most) one jump across level $y$, while local time $\ell^y$ remains constant during each excursion. Therefore,
inverse local time $\tau^y$ is a natural time scale for processes $t\mapsto m_h(y,t)$, and indeed for $t\mapsto \markZ_{[0,t]}(y)$:

\begin{proposition}\label{propcmjstable} In the setting of Theorem \ref{thm1}, we have for each $y\in\bR$ the following subordinators.
  \begin{enumerate}\item[(i)] The process $s\mapsto \markZ_{[0,\tau^y(s)]}(y)-Z_{[0,\tau^y(0)]}(y)$ is a stable subordinator with index
    $\stablebetaminusonewith /q$, for each $y\in\bR$. Its Laplace exponent is  $\Theta(\xi)=c\Gamma(1-\stablebetaminusonewith /q)\xi^{\stablebetaminusonewith /q}$.
  \item[(ii)] The process $s\mapsto m_h(y,\tau^y(s))-m_h(y,\tau^y(0))$ is a Poisson process of rate $ch^{-\stablebetaminusonewith /q}$ independent of 
    $m_h(y,\tau^y(0))=1{\{Z_{T_{\ge y}}(y-X_{T_{\ge y}-})>h\}}$, where $T_{\ge y}=\inf\{t>0\colon X_t\ge y\}$.
  \end{enumerate} 
\end{proposition}

In particular, we deduce from (i) that $(\markZ_{[0,\tau^y(s)]}(y),0\le s<\ell^y(T))$ is a stopped stable subordinator. It is well-known \cite{csp,PitmYorPDAT} that its 
ranked relative jump sizes have Poisson-Dirichlet distribution, and that its $\diversityalpha$-diversity (for $\diversityalpha=\stablebetaminusonewith /q$), a proxy for 
genetic diversity, can be approximated by the number of jumps/blocks exceeding threshold $h$, i.e. by $m_h(y,T)$. Hence, Theorem \ref{thm1} shows that this 
$\diversityalpha$-diversity process coincides with the local time process of $X$, and Theorem \ref{thmltholdermoment} is a statement on the H\"older-continuity of the 
$\diversityalpha$-diversity process as a function of $y$, expressing the evolution of the diversity of subpopulations $\markZ_t(y-X_{t-})$, $0\le t\le T$, in 
$\markZ_{[0,T]}(y)$. See Ruggiero et al. \cite{RWF2013} for another recent study that exhibits continuous diversity processes.

A further motivation for this study lies in the construction of diffusion processes in spaces of interval partitions and $\bR$-trees, which we explore in 
forthcoming work, see \cite{Paper1}. 

This paper is structured, as follows. 
In Section \ref{secprel}, we discuss some fluctuation theory of L\'evy processes and prove Proposition \ref{propunif}. 
In Section \ref{secunifapprox}, we prove Proposition 
\ref{propuniformholder}, Proposition \ref{propcmjstable}, Theorem \ref{thm1} and Corollary \ref{corhat}. 
In Section \ref{secheight}, we establish Proposition \ref{proppiling} and Corollary \ref{corcmj} using a local time approximation by jump heights, which is uniform on a
sequence of refining spatial lattices in regions where local time is bounded below. Section \ref{sect2} establishes H\"older constants with all moments for squared 
Bessel processes, excursions and bridges, and related processes. This provides a class of examples to which the main results of this paper apply. Finally, Section
\ref{sectrestricted} studies spectrally positive L\'evy processes restricted to an interval and proves Theorem \ref{thmltholdermoment}.

\section{Preliminaries on spectrally positive stable L\'evy processes}\label{secprel}

\noindent Recall the \em occupation density formula \em that characterises the local times $(\ell^y(t),y\!\in\!\bR,t\!\ge\! 0)$ of $X$ 
$$\int_0^T f(X_t)dt=\int_{-\infty}^\infty f(y)\ell^y(T)dy\quad a.s.,\qquad\mbox{for all }f\colon\bR\rightarrow\bR\mbox{ bounded measurable, }T\ge 0.$$
The following lemma is easily obtained using the positivity parameter $1-1/\stablebetawith$, spatial homogeneity and scaling 
properties of $X$. See e.g.\ Bertoin \cite[Sections V.1,VIII.1]{Ber1996} for background.

\begin{lemma}\label{tauzero} The Laplace exponent of the inverse local time $\tau^0$ of $X$ is $\stablebetawith\xi^{1-1/\stablebetawith}$.
\end{lemma}

  Denote by $n$, $\widehat{n}$ and $\underline{\widehat{n}}$ the excursion measures of $X$, $\widehat{X}=-X$ and
  $\widehat{X}-\widehat{I}=S-X$ away from $0$, where $\widehat{I}_t=\inf\{\widehat{X}_s,s\le t\}$ and $S_t=\sup\{X_s,s\le t\}$, $t\ge 0$, are the running infimum and
  supremum of $\widehat{X}$ and $X$, respectively. We denote by $H(\omega)=A(\omega)+B(\omega)$ the size of the jump of an excursion $\omega$ across level 0, where
  $A(\omega)=-\omega(T_{\ge 0}-)$ and $B(\omega)=\omega(T_{\ge 0})$ are the undershoot and overshoot at the crossing time 
  $T_{\ge 0}(\omega)=\inf\{t>0\colon\omega(t)\ge 0\}$. Similarly, $\widehat{A}(\omega)=\omega(T_{\le 0}-)$ and $\widehat{B}(\omega)=-\omega(T_{\le 0})$ for
  $T_{\le 0}(\omega)=\inf\{t>0\colon\omega(t)\le 0\}$. We are interested in the relative undershoot $U=A/H$ jointly with $H$ under $n$, equivalently 
  $\widehat{U}=\widehat{A}/\widehat{H}$ jointly with $\widehat{H}=\widehat{A}+\widehat{B}$ under $\widehat{n}$. Under $\underline{\widehat{n}}$, the relevant crossing is 
  the end of the excursion, at $T_0(\omega)=\inf\{t>0\colon\omega(t)=0\}$. 

\begin{proof}[Proof of Proposition \ref{propunif}] 
By \cite[Theorem 3]{PPR15arxiv}, we have
  $$\widehat{n}\left(1_{\{t<T_{\le 0}\}}f\left(\omega\left(T_{\le 0}-\right),-\omega\left(T_{\le 0}\right)\right)\right)
	=\underline{\widehat{n}}\left(1_{\{t<T_0\}}\int f(y,-z)K\left(\omega(t),dy,dz\right)\right),$$
  where $K(x,dy,dz)=(W(x)-W(x-y))\Pi(dz-y)dy$ on $(-\infty,0)\times(0,\infty)$, with $W(x)=(\Gamma(1+\alpha))^{-1}x^\alpha$, $x>0$, the scale function of $X$, and
  $\Pi(dx)=\stablebetawith\stablebetaminusonewith(\Gamma(1-\stablebetaminusonespecial))^{-1}x^{-\alpha-2}dx$ the L\'evy measure of $X$. We also denote by $u(t,x)=p_x(t)$ 
  the bivariate renewal density of the decreasing ladder processes of $X$, where $p_y(t)$ is the density of $\sigma_y$ at $t$ for a $\sigma=(\sigma_y,y\ge 0)$ with 
  Laplace exponent $\Phi(\xi)=\xi^{1/\stablebetawith}$. Specifically, we can read from \cite[Sections 7.3, 8.1 and 8.2]{Kyprianou} that 
  $\int_0^\infty e^{-\eta y}W(y)dy=1/\eta^{\stablebetawithout}$ and $\int_0^\infty\int_0^\infty e^{-\xi t-\eta x}u(t,x)dxdt=1/(\xi^{1/\stablebetawith}+\eta)$, and these 
  Laplace transforms are easily inverted. From \cite[Equation (5.2.5)]{Don06}, we take the entrance law $\underline{\widehat{n}}(t<T_0,\omega(t)\in dx)=c^\prime u(t,x)dx$,
  up to a constant $c^\prime\in(0,\infty)$. Hence,
  \begin{align*}&\underline{\widehat{n}}\left(1_{\{t<T_0\}}\int f(y,-z)K(\omega(t),dy,dz)\right)\\
                &=\frac{c^\prime\alpha(1+\alpha)\Gamma(1+\alpha)}{\Gamma(1-\alpha)}\int_0^\infty p_x(t)
					   	\Bigg(\int_0^x(x^\alpha-(x-y)^\alpha)\int_{-\infty}^0f(y,-z)(y-z)^{-\alpha-2}dzdy\\
			    &\hspace{7cm}+\int_x^\infty x^{\alpha}\int_{-\infty}^0f(y,-z)(y-z)^{-\alpha-2}dzdy\Bigg)dx.
  \end{align*}
  To compute this, and to let $t\downarrow 0$, we apply Fubini's theorem and first consider the $x$-integral over the $(x-y)^{\alpha}$-term:
  \begin{align*}\int_y^\infty p_x(t)(x-y)^\alpha dx&=\int_0^\infty p_{r+y}(t)r^\alpha dr=\int_0^\infty p_1((r+y)^{-\alpha-1}t)(r+y)^{-\alpha-1}r^\alpha dr\\
												   &=t^{-\alpha/(1+\alpha)}\frac{1}{1+\alpha}\int_0^{ty^{-1-\alpha}}p_1(s)s^{-1/(1+\alpha)}\left(t^{1/(1+\alpha}s^{-1/(1+\alpha)}-y\right)^\alpha ds\\
												   &\le\frac{1}{1+\alpha}\int_0^\infty 1_{\{s\le ty^{-1-\alpha}\}}p_1(s)s^{-1}ds,
  \end{align*}
  by the Dominated Convergence Theorem, as the integral without the indicator equals $\Gamma(1/\alpha)/\alpha$ (see e.g. \cite[Equation (0.40)]{csp}). The 
  remaining terms yield, using similar substitutions,
  \begin{align*} &\frac{c^\prime\alpha(1+\alpha)\Gamma(1+\alpha)}{\Gamma(1-\alpha)}\frac{1}{1+\alpha}\int_0^\infty p_1(s)s^{-1}ds\int_0^\infty\int_{-\infty}^0
       f(y,-z)(y-z)^{-\alpha-2}dzdy\\
     &=\frac{c^\prime\Gamma(1+\alpha)\Gamma(1/\alpha)}{\Gamma(1-\alpha)}\int_0^\infty\int_0^{\infty}f(y,z)(y+z)^{-\alpha-2}dzdy.
  \end{align*} 
  We conclude that under the excursion measure $n$ of $X=-\widehat{X}$, undershoot $A$ and overshoot $B$ satisfy
  $$n(A\in dy,B\in dz)=\frac{c^\prime\Gamma(1+\alpha)\Gamma(1/\alpha)}{\Gamma(1-\alpha)}(y+z)^{-\alpha-2}dzdy$$
  and hence $H=A+B$ and $U=A/H$ satisfy
  $$n(H\in dx,U\in dr)=\frac{c^\prime\Gamma(1+\alpha)\Gamma(1/\alpha)}{\Gamma(1-\alpha)}x^{-\alpha-1}dxdr.$$
  In particular, $U=A/(A+B)$ is uniformly distributed under $n(\,\cdot\,|\,H=x)$ for a.e. $x>0$ and indeed for all $x>0$ by the scaling property. 
  The remaining results are classical. The independence claim is a consequence of the Markov property under $n$, see e.g. \cite[Section IV.4]{Ber1996}, as is the 
  conditional distribution of the post-jump process under $n(\,\cdot\,|\,A=y,B=z)$. The claim about the pre-jump process follows by time-reversal \cite{GetoShar81}. Since
  the two first hitting times are downward level passage times, their sum has Laplace transform given by 
  $$n\left(\exp(-\xi T_0)\,|\,A=y,B=z\right)=\exp\left(-(y+z)\Phi(\xi)\right).$$
  On the other hand, by Lemma \ref{tauzero}, the Laplace exponent of
  $\tau^0$ is $\stablebetawith\xi^{1-1/\stablebetawith}$. Therefore 
$$n\left(1-\exp\left(-\xi T_0\right)\right)=\stablebetawith\xi^{1-1/\stablebetawith}=\int_0^\infty\left(1-\exp\left(-x\xi^{1/\stablebetawith}\right)\right)\frac{\stablebetawith\stablebetaminusonewith}{\Gamma(1-\stablebetaminusonespecial)}x^{-\stablebetaminusonespecial-1}dx,$$
  so we require
  $$\frac{c^\prime\Gamma(1+\alpha)\Gamma(1/\alpha)}{\Gamma(1-\alpha)}x^{-\alpha-1}
    =\frac{\stablebetawith\stablebetaminusonewith}{\Gamma(1-\stablebetaminusonespecial)}x^{-\stablebetaminusonespecial-1},
    \quad\mbox{i.e. }c^\prime=\frac{\stablebetawith\stablebetaminusonewith}{\Gamma(1+\alpha)\Gamma(1/\alpha)}.$$          
\end{proof}

\section{Uniform approximation of local times}\label{secunifapprox}

We will use notation $\widehat{m}_h(y,r)=m_h(y,\tau^y(r))-m_h(y,\tau^y(0))$. 
The purpose of this section is to prove Theorem \ref{thm1}. The argument is inspired by Khoshnevisan \cite{Kho-94}, 
who studies rates of convergence. We only need weaker results. Weaker 
results have also been obtained by different methods for different classes of processes e.g. by \cite{BPT1,BPT2,CRe-86,Perkins}, but Khoshnevisan's methods seem most 
adaptable to our setting. Khoshnevisan uses excursion theory to study intrinsic approximations of Brownian local times based on excursion lengths. Excursion lengths 
while intrinsic to the level set are not intrinsic to the population sizes $(Z_t(y-X_{t-})$, $0\le t\le T$, at level $y$, while our approximations in Theorem \ref{thm1}
have this latter property.

\subsection{Proofs of Proposition \ref{propuniformholder} and Proposition \ref{propcmjstable}}

\begin{proof}[Proof of Proposition \ref{propuniformholder}] First note that $\kappa_q(x,\cdot)$ is the distribution of $(x^q\markZ(s/x),0\le s\le x)$ with H\"older 
  constant
  $$\sup_{0\le r<s\le x}\frac{|x^q\markZ(s/x)-x^q\markZ(r/x)|}{|s-r|^\pathholder}=x^{q-\pathholder}D_\pathholder.$$
  For all $t\in[0,T]$ with $\Delta X_t>0$, denote the associated H\"older constant by 
  $$d_t=\sup_{0\le r<s\le\Delta X_t}\frac{|\markZ_t(s)-\markZ_t(r)|}{|s-r|^\pathholder},$$
  and set $d_t=0$ if $\Delta X_t=0$.
  Now clearly for all $p>0$,
  $$\left(\sup\{d_t\colon 0\le t\le T\}\right)^p\le\max\left\{\left(\sup\{d_t^p\colon 0\le t\le T,\Delta X_t>1\}\right)^p,\sum_{0\le t\le T\colon\Delta X_t\le 1}d_t^p\right\},$$
  and by the compensation formula for Poisson point processes,
  $$\bE\left(\sum_{0\le t\le T\colon\Delta X_t\le 1}d_t^p\right)= T\frac{\stablebetawith \stablebetaminusonewith }{\Gamma(1-\stablebetaminusonespecial)} \bE((D_\pathholder)^p)\int_0^1 x^{p(q-\pathholder)-\stablebetaminusonespecial   -2}dx<\infty,$$
  provided that $p>\stablebetawith /(q-\pathholder)$.  This completes the proof, since there are at most finitely many jumps of size exceeding 1 in $[0,T]$, as their rate is finite in the Poisson point process.
\end{proof}

Using the same argument, we also obtain the following result.

\begin{lemma}\label{cor:all_short_excs_holder}
 Let $\theta\in (0,\gamma]$ and let $\widehat d_t$ denote the $\theta$-H\"older constant of $Z_t$
 \begin{align}
  \widehat{d}_t := \sup_{0\le r<s\le\Delta X_t}\frac{|Z_t(s)-Z_t(r)|}{(s-r)^{\theta}}.
 \end{align}
 For every $\varepsilon > 0$ there is some non-random $C_{\varepsilon}>0$ such that for every $z > 0$,
 \begin{align}
  \bE\left( \sup_{0\le t\le T\colon 0<\Delta X_t\le z} \widehat d_t \right) < C_{\varepsilon}z^{q - \theta - \varepsilon}.\label{eq:all_short_excs_holder}
 \end{align}
\end{lemma}

\begin{proof}
 Fix $\theta\in (0,\gamma]$ and let $p > (1+\alpha)/(q-\theta)$. We apply Jensen's inequality and argue
 as in the proof of Proposition \ref{propuniformholder}, via the scaling invariance of $\kappa_q$, to find that
 \begin{align*}
  &\bE\left( \sup_{0\le t\le T\colon \Delta X_t\le z} \widehat{d}_t \right)\leq \left(\bE\left( \sum_{0\le t\le T\colon \Delta X_t\le z} {\widehat{d}}_t^{\;p} \right)\right)^{1/p}\\
  	&= \left(\bE\left( \sum_{0\le t\le T\colon \Delta X_t\le z} (\Delta X_t)^{(q-\theta)p} \right)\right)^{1/p}\left(\bE(D_\theta^p)\right)^{1/p}\\
  	&= \left(\bE(D_\theta^p)\right)^{1/p}\left(\int_0^{z}x^{(q-\theta)p}\frac{(1+\alpha)\alpha T}{\Gamma(1-\alpha)}x^{-\alpha-2}dx\right)^{1/p}
  	\!= \left(\bE(D_\theta^p)\right)^{1/p}C'_p z^{q-\theta - (1+\alpha)/p},
 \end{align*}
 where $C'_p$ is a finite deterministic term that depends on $p$. Since $D_\theta$ has moments of all orders, the above expression is finite for all 
 $p > (1+\alpha)/(q-\theta)$. To obtain \eqref{eq:all_short_excs_holder}, we take $p > (1+\alpha)/\varepsilon$.
\end{proof}

\begin{proof}[Proof of Proposition \ref{propcmjstable}] Recall notation $n$ for the It\^o excursion measure of $X$ away from 0, $H$ for the size of the jump across 0, 
  and $U$ for the relative undershoot of that jump across 0. We applied the marking kernel $\kappa_q$ to the Poisson point process of all jumps of $X$, and this induces
  a $\kappa_q$-mark $Z_{T_{\ge 0}}$ of the jump $(T_{\ge 0},H)$ under $n$. We write $n_+(d\omega,df)=\kappa_q(H(\omega),df)n(d\omega)$ for the intensity measure of 
  the It\^o excursion process with marked jump across 0. Recall the joint density of $(H,U)$ under $n$ from Proposition \ref{propunif}.

  By \eqref{cmj}, the jumps of $t\mapsto Z_{[0,t]}(0)$ will be of the form $J_t=Z_t(-X_{t-})$, where $-X_{t-}$ is the undershoot of any jump of $X$ at time $t$ across
  level 0. Since $X$ is spectrally positive, there is at most one such jump per excursion. A standard mapping argument shows that those jumps form a Poisson point process 
  in the inverse local time parametrisation of the It\^o excursion process, with intensity measure
  \begin{align*}n_+(\{(\omega,f)\colon f(H(\omega)U(\omega))>h\})
	&=\int_0^\infty\int_0^1\bP(x^qZ(u)>h)\frac{\stablebetawith\stablebetaminusonewith}{\Gamma(1-\stablebetaminusonespecial)}x^{-\stablebetaminusonespecial-1}dxdu\\
    &=\int_0^\infty\bP(Z(U)>z)\frac{\stablebetawith\stablebetaminusonewith}{\Gamma(1-\stablebetaminusonespecial)q}h^{-\stablebetawith/q+1/q}z^{\stablebetawith/q-1/q-1}dz\\
    &=\frac{\stablebetawithout}{\Gamma(1-\stablebetaminusonespecial)}h^{-\stablebetaminusonewith/q}\bE((Z(U))^{\stablebetaminusonewith/q})
    =ch^{-\stablebetaminusonewith/q},
  \end{align*}
  identifying $c$ as given in the statement of Theorem \ref{thm1}. We note that this establishes part (ii) of Proposition \ref{propcmjstable}. For (i), we obtain 
  that the $J_t$ are summable for $q>\stablebetaminusonewithout$, as we recognise the tail of the stable L\'evy measure of index $\stablebetaminusonewith/q$, with 
  Laplace exponent
  $$\Theta(\xi)=\int_{(0,\infty)}(1-e^{-\xi h})n_+(f(H(\omega)U(\omega))\in dh)=c\Gamma(1-\stablebetaminusonewith/q)\xi^{\stablebetaminusonewith/q}.$$
  The generalisation from $y=0$ to general $y\in\bR$ follows by spatial homogeneity of $X$ and by the strong Markov property of $X$ at $T_y=\inf\{t>0\colon X_t=y\}$,
  since $T_y\ge T_{\ge y}$ a.s.. 
\end{proof}

\subsection{Auxiliary results for the proof of Theorem \ref{thm1}}

\noindent To prove Theorem \ref{thm1}, we will need a moderate deviations result for Poisson processes, which we deduce from standard large deviations results as can be found e.g.\ in Dembo and Zeitouni \cite{DeZ}.

\begin{lemma}\label{lm6} Let $(N_t,t\ge 0)$ be a standard Poisson process with $N_t\sim{\rm Poi}(t)$. Then for all $z>\delta>0$, there is $t_0\ge 1$ such that for all $t\ge t_0$
  $$\bP\left(\left|N_{t}-t\right|\ge\sqrt{t\log(t)}\sqrt{2z}\right)\le t^{-z+\delta}.$$ 
\end{lemma}
\begin{proof} We apply \cite[Theorem 3.7.1]{DeZ} to independent centred Poisson variables $Y_i=N_i-N_{i-1}-1$, $i\ge 1$ and $a_n=(2z\log(n))^{-1}$, $n\ge 1$. Since $a_n\rightarrow 0$ and $na_n\rightarrow\infty$ as $n\rightarrow\infty$, we find that for any $\varepsilon>0$, there is $n_0\ge 1$ such that for all $n\ge n_0$
$$\bP\left(\sum_{i=1}^nY_i\ge\sqrt{\frac{n}{a_n}}\right)\le \exp\left(-\frac{1}{2a_n}\right)=n^{-z}.$$
To pass from integer $n\ge 1$ to real $t\ge 2$, we note that
$$\bP\left(N_t-t\ge\sqrt{t\log(t)}\sqrt{2z}\right)\le\bP\left(N_{\lceil t\rceil}-\lceil t\rceil\ge-1+\sqrt{(\lceil t\rceil-1)\log(\lceil t\rceil -1)}\sqrt{2z}\right)$$
is of the same form, with $\sqrt{n/a_n}$ replaced by $\sqrt{n/b_n}=-1+\sqrt{(n-1)\log(n-1)}\sqrt{2z}$, and another application of \cite[Theorem 3.7.1]{DeZ} yields
$$\bP\left(N_t-t\ge\sqrt{t\log(t)}\sqrt{2z}\right)\le\exp\left(-\frac{1}{2b_n}\right)\le {\lceil t\rceil}^{-z+\delta/2}\le t^{-z+\delta/2},$$
for $t\!\ge\! n_0$, for a possibly increased $n_0$. A similar argument deals with $\bP(t-N_t\!\ge\!\sqrt{t\log(t)}\sqrt{2z})$, and together, possibly increasing $n_0$ again, we obtain the stated result.
\end{proof}

\noindent By Proposition \ref{propcmjstable}, we have $\widehat{m}_h(y,r)\sim{\rm Poi}(rch^{-\stablebetaminusonewith/q})$. Moreover, 
$t\mapsto N_t(y,r)=\widehat{m}_{(rct^{-1})^{q/\stablebetaminusonewith}}(y,r)$, $t>0$ is a standard 
Poisson process, by independence properties of the Poisson point process of jumps of $s\mapsto Z_{[0,\tau^y(s)]}(y)$. We will apply the previous lemma to a variation of 
this Poisson process. Following Khoshnevisan \cite[Lemmas 5.1-5.3]{Kho-94}, we consider an independent $\lambda\sim{\rm Exp}(1)$:

\begin{lemma}\label{lm7} Let $T_y=\inf\{t\ge 0\colon X_t=y\}$, $y\in\bR$. Then for all $y>0$, we have
  \begin{align*}&\bP(T_{-y}<\lambda)=\exp(-y),\\
  \mbox{and}\qquad&\bP(T_y<\lambda)=\frac{\stablebetawithout }{\pi}\int_0^\infty\frac{\sin(\pi \stablebetaminusonewith ) s^{\stablebetawithout } }{s^{2\stablebetawith }+2s^{\stablebetawithout } \cos(\pi\stablebetaminusonewith )+1} e^{-ys}ds .
\end{align*}
\end{lemma}
\begin{proof} Let $y>0$. Since $X$ is spectrally positive, we have $T_{-y}=\inf\{t\ge 0\colon X_t\le -y)$, and it is well-known that $\bP(T_{-y}<\lambda)=\bE(e^{-T_{-y}})=\exp(-y)$, see e.g. Bertoin \cite[Chapter VII]{Ber1996} for first passage problems of spectrally negative L\'evy processes such as $-X$. On the other hand, Simon \cite{Sim-11} showed that, $T_1\sim RT_{-1}$, where $R$ is independent of $T_{-1}$ with probability density function 
\[f_R(t) = \frac{\sin(\pi \stablebetaminusonewith ) t^{1/\stablebetawith }}{\pi(t^2+2t\cos(\pi\stablebetaminusonewith )+1)}1_{\{t\ge 0\}}.\]
By scaling, $T_y\sim y^{\stablebetawithout  }T_1\sim y^{\stablebetawithout  }RT_{-1}\sim RT_{-y}$. With this, we obtain
\[\begin{split} \bP(T_y<\lambda) = \bP(T_{-y}R<\lambda) & =\frac{1}{\pi}\int_0^\infty\frac{\sin(\pi \stablebetaminusonewith ) t^{1/\stablebetawith  }}{t^2+2t\cos(\pi\stablebetaminusonewith )+1}\bP\left(tT_{-y}<\lambda\right)dt \\
& =\frac{1}{\pi}\int_0^\infty\frac{\sin(\pi \stablebetaminusonewith ) t^{1/\stablebetawith  }}{t^2+2t\cos(\pi\stablebetaminusonewith )+1}\bP\left(T_{-yt^{1/\stablebetawith }}<\lambda\right)dt\\
& = \frac{1}{\pi}\int_0^\infty\frac{\sin(\pi \stablebetaminusonewith ) t^{1/\stablebetawith  }}{t^2+2t\cos(\pi\stablebetaminusonewith )+1} e^{-yt^{1/\stablebetawith }}dt \\
& =\frac{\stablebetawithout }{\pi}\int_0^\infty\frac{\sin(\pi \stablebetaminusonewith ) s^{\stablebetawithout } }{s^{2\stablebetawith }+2s^{\stablebetawithout } \cos(\pi\stablebetaminusonewith )+1} e^{-ys}ds 
.\end{split}\]\vspace{-0.5cm}

\end{proof}

\begin{lemma} For all $y\in\bR$, we have  
  $$\bP(\ell^y(\lambda)>r)=\bP(T_y<\lambda)\exp\left(-\stablebetawith r\right).$$
\end{lemma}
\begin{proof}  By the strong Markov property and spatial homogeneity of $X$, we have
  $$\bP(\ell^y(\lambda)>a)=\bP(T_y<\lambda)\bP(\ell^0(\lambda)>a).$$
  To calculate $\bP(\ell^0(\lambda)>a)$, note that by the strong Markov property of $X$ at inverse local times, $\ell^0(\lambda)$ is exponentially distributed, and by
  Lemma \ref{tauzero},
  \begin{align*}\bE(\ell^0(\lambda))&=\int_0^\infty\bE\left(\ell^0(t)\right)e^{-t}dt=\int_0^\infty\int_0^\infty\bP\left(\ell^0(t)>s\right)dse^{-t}dt\\
    &=\int_0^\infty\int_0^\infty\bP\left(\tau^0(s)<t\right)e^{-t}dtds=\int_0^\infty\bE\left(e^{-\tau^0(s)}\right)ds=\frac{1}{\stablebetawithout},
  \end{align*}
  i.e. $\ell^0(\lambda)$ has distribution ${\rm Exp}(\stablebetawithout)$, as claimed.
\end{proof}

Recall notation $\widehat{m}_h(y,r)=m_h(y,\tau^y(r))-m_h(y,\tau^y(0))$ for the Poisson counting process of level-$y$ excursions of $X$ with path-mark exceeding $h$ at
the crossing level, parametrised by level-$y$ local time $r\ge 0$. Recall from Proposition \ref{propcmjstable}(ii) that its rate is $ch^{-\alpha/q}$. Note that for each $y\in\bR$, the
inverse local time $\tau^y(0)$ equals the first hitting time $T_y$ of level $y$ a.s., and that $m_h(y,\tau^y(0))=1$ if the single jump of $X$ across level $y$ before 
$T_y$ has a path-mark exceeding $h$, while $m_h(y,\tau^y(0))=0$ otherwise. 

We will study $m_h(y,\lambda)$ by first investigating $\widehat{m}_h(y,\ell^y(\lambda)-)$, i.e. the excursion count stopped just before the excursion 
straddling the independent exponential time $\lambda$. On the event $\{\lambda\le T_y\}$, we have $\ell^y(\lambda)=0$ and $\widehat{m}_h(y,\ell^y(\lambda)-)=0$. 
On the event $\{\lambda>T_y\}$, we consider $\lambda$ as a further mark of the excursion of $X$ straddling $\lambda$, and we can represent the stopped
excursion count as a thinned Poisson process $r\mapsto\widetilde{m}_h(y,r)$ stopped at an independent time $\ell^y(\lambda)\sim{\rm Exp}(\stablebetawithout)$. The
thinned Poisson process has rate $ch^{-\alpha/q}-a_h$ for some $a_h\in[0,\stablebetawithout]$. See e.g. \cite[Section VI.49]{RoW2} for details on the thinning of
excursion processes. 

Finally, note that $m_h(y,\lambda)-m_h(y,\tau^y(0))-\widehat{m}_h(y,\ell^0(\lambda)-)=1$ if the excursion straddling $\lambda$ has crossed before $\lambda$ with a
path-mark exceeding $h$ at the crossing level, while $m_h(y,\lambda)-m_h(y,\tau^y(0))-\widehat{m}_h(y,\ell^0(\lambda)-)=0$ otherwise. In summary, we have for each $y\in\bR$
\begin{equation}\label{mmhat} \widehat{m}_h(y,\ell^y(\lambda)-)\le m_h(y,\lambda)\le\widehat{m}_h(y,\ell^y(\lambda)-)+2\qquad a.s.
\end{equation}  

\begin{lemma}\label{lm11} For all $r_0>0$, $\theta>0$, $\varepsilon>0$, there is $h_0>0$ such that for all $h\le h_0$, $y\in\bR$,
  $$\bP\left(\left|\frac{h^{\stablebetaminusonewith /q}}{c}\widehat{m}_h(y,\ell^y(\lambda)-)-\ell^y(\lambda)\right|\ge \sqrt{\frac{2(\theta\!+\!\varepsilon)(1\!+\!2\varepsilon)}{c}}h^{\stablebetaminusonewith /2q}\sqrt{\log\left(\frac{1}{h}\right)}\sqrt{\ell^y(\lambda)},\ell^y(\lambda)\!\ge\! r_0\!\right)\le h^\theta.$$
\end{lemma}
\begin{proof} Recall that $q>\stablebetaminusonewithout $. Let $r_0>0$ and $\theta>0$ and, without loss of generality, $\varepsilon>0$ so small that we have 
  $\varepsilon q(\theta+\varepsilon)(2+7\varepsilon^2+4\varepsilon^4)<q-\stablebetaminusonewith $. Then we apply the strong Markov property of $X$ at $T_y$ and the spatial homogeneity of $X$ to find
that
  \begin{eqnarray*}p_h\!\!\!\!&:=&\!\!\!\bP\left(\left|\frac{h^{\stablebetaminusonewith /q}}{c}\widehat{m}_h(y,\ell^y(\lambda)-)-\ell^y(\lambda)\right|\ge \sqrt{\frac{2(\theta\!+\!\varepsilon)(1\!+\!2\varepsilon)}{c}}h^{\stablebetaminusonewith /2q}\sqrt{\log\left(\frac{1}{h}\right)}\sqrt{\ell^y(\lambda)},\ell^y(\lambda)\!\ge\! r_0\right)\\
    &=&\!\!\!\bP(T_y<\lambda)\bP\Bigg(\left|\frac{h^{\stablebetaminusonewith /q}}{c}\widehat{m}_h(0,\ell^0(\lambda)-)-\ell^0(\lambda)\right|\\
       &&\qquad\qquad\qquad\ge \sqrt{\frac{2(\theta+\varepsilon)(1+2\varepsilon)}{c}}h^{\stablebetaminusonewith /2q}\sqrt{\log\left(\frac{1}{h}\right)}\sqrt{\ell^0(\lambda)},\ell^0(\lambda)\ge r_0\Bigg)\\
    &\le&\!\!\!\int_{r_0}^\infty\bP\left(\left|\frac{h^{\stablebetaminusonewith /q}}{c}\widetilde{m}_h(0,r)-r\right|\ge \sqrt{\frac{2(\theta\!+\!\varepsilon)(1\!+\!2\varepsilon)}{c}}h^{\stablebetaminusonewith /2q}\sqrt{\log\left(\frac{1}{h}\right)}\sqrt{r}\right)\bP(\ell^0(\lambda)\!\in\! dr),
  \end{eqnarray*}
  where $\widetilde{m}_h(0,\cdot)$ is the thinning of $\widehat{m}_h(0,\cdot)$ independent of $\ell^0(\lambda)$ as discussed above the statement of the lemma. As a function of $h>0$, for fixed $r$, we can write 
  $\widetilde{m}_h(0,r)=N(r(ch^{-\stablebetaminusonewith /q}-a_h))$ in terms of a unit rate Poisson 
  process $N$, to which Lemma \ref{lm6} applies. Due to the thinning, we need to re-centre to get a deviation probability from the mean in the integrand above:
  \begin{align*}&\bP\left(\left|N(r(ch^{-\stablebetaminusonewith /q}-a_h))-rch^{-\stablebetaminusonewith /q}\right|
				   \ge\sqrt{rch^{-\stablebetaminusonewith /q}\log\left(\frac{1}{h}\right)}\sqrt{\frac{2(\theta+\varepsilon)(1+2\varepsilon)}{c}}\right)\\
                &\le\bP\left(\left|N\!\left(r\!\left(ch^{-\stablebetaminusonewith /q}\!-\!a_h\right)\!\right)-r\!\left(ch^{-\stablebetaminusonewith /q}\!-\!a_h\right)\right|
\ge\sqrt{rch^{-\stablebetaminusonewith /q}\log\left(\frac{1}{h}\right)}\sqrt{\frac{2(\theta\!+\!\varepsilon)(1\!+\!2\varepsilon)}{c}}-ra_h\!\right).
  \end{align*}
  The deviation threshold is not quite of the form to which Lemma \ref{lm6} applies, either, but after a few more steps,
  we will apply Lemma \ref{lm6} for $z:=(\theta+\varepsilon)(1-2\varepsilon^2-7\varepsilon^4-4\varepsilon^6)q/\stablebetaminusonewith $ and
  $\delta:=z-\theta q/\stablebetaminusonewith -\varepsilon=\varepsilon(q/\stablebetaminusonewith -1-\varepsilon(\theta+\varepsilon)(2+7\varepsilon^2+4\varepsilon^4)q/\stablebetaminusonewith )$, which satisfy 
  $z>\delta>0$ by the restriction we put onto $\varepsilon$. Let $t_0$ be obtained from Lemma \ref{lm6} and set
  $$h_0\!:=\!\min\left\{\!\frac{1}{e},\!
                    \left(\frac{\sqrt{c}}{\stablebetawithout }\varepsilon^2\sqrt{2(\theta\!+\!\varepsilon)(1\!+\!2\varepsilon)}\right)^{\!2(1\!-\!2\varepsilon)q/\stablebetaminusonewith (1\!-\!3\varepsilon)}\!\!,\!
                    \left(\frac{cr_0}{\stablebetawith  r_0\!+\!t_0}\right)^{\!q/\stablebetaminusonewith }\!,
                    c^{-(1-2\varepsilon)q/\stablebetaminusonewith \varepsilon}\right\}.$$
  Now let $r\le h^{-\varepsilon\stablebetaminusonewith /q(1-2\varepsilon)}$. First, we turn the $a_hr$ term on the right-hand side into $\varepsilon^2$ times the square-root term: for 
  all $h\le h_0$
  $$a_h\le\stablebetawithout ,\quad\sqrt{r}\le h^{-\varepsilon\stablebetaminusonewith /2q(1-2\varepsilon)}\quad\mbox{and}\quad 1\le\log(1/h),$$
  using the first bound on $h_0$ and find that our claim is equivalent to the second bound on $h_0$. Then we check that by the third bound on $h_0$
  $$c\le h^{-\varepsilon\stablebetaminusonewith /q(1-2\varepsilon)}
	\quad\mbox{and hence}\quad\log\left(\frac{1}{h}\right)\ge\frac{q}{\stablebetaminusonewithout }(1-2\varepsilon)\log\left(rch^{-\stablebetaminusonewith /q}\right),$$
  for all $h\le h_0$. Finally, we estimate $rch^{-\stablebetaminusonewith /q}\ge rch^{-\stablebetaminusonewith /q}-ra_h$ and note that $rch^{-\stablebetaminusonewith /q}-ra_h\ge t_0$ for all $h\le h_0$ is
  equivalent to the fourth bound on $h_0$. Hence, Lemma \ref{lm6} applies and yields an upper bound for our integrand of
  \begin{align*}&\bP\Bigg(\left|N\left(r\left(ch^{-\stablebetaminusonewith /q}-a_h\right)\right)-r\left(ch^{-\stablebetaminusonewith /q}-a_h\right)\right|\\
				   &\qquad\qquad\ge\sqrt{r\left(ch^{-\stablebetaminusonewith /q}-a_h\right)\log\left(r\left(ch^{-\stablebetaminusonewith /q}-a_h\right)\right)}
                      \sqrt{\frac{2(\theta+\varepsilon)(1-4\varepsilon^2)(1+\varepsilon^2)^2}{c}}\Bigg)\\
				&\le\left(r\left(ch^{-\stablebetaminusonewith /q}-a_h\right)\right)^{-\theta q/\stablebetaminusonewith -\varepsilon}
                 \le\left(c-\stablebetawith \sqrt{h_0}\right)^{-\theta q/\stablebetaminusonewith -\varepsilon}h^{\theta+\varepsilon\stablebetaminusonewith /q} r^{-\theta q/\stablebetaminusonewith -\varepsilon},
  \end{align*}
  so that $p_h$ is bounded above by
  \begin{align*}&\!\left(c\!-\!\stablebetawith \sqrt{h_0}\right)^{-\theta q/\stablebetaminusonewith -\varepsilon}h^{\theta+\varepsilon\stablebetaminusonewith /q}
							    \int_{r_0}^{h^{-\varepsilon\stablebetaminusonewith /q(1-2\varepsilon)}}\!\!r^{-\theta q/\stablebetaminusonewith -\varepsilon}\bP(\ell^0(\lambda)\in dr)
						   +\bP\left(\ell^0(\lambda)>h^{-\varepsilon\stablebetaminusonewith /q(1-2\varepsilon)}\right)\\
    &\ \ \le\left(r_0c-r_0\stablebetawith \sqrt{h_0}\right)^{-\theta q/\stablebetaminusonewith -\varepsilon}h^{\theta+\varepsilon\stablebetaminusonewith /q}+\exp\left(-\stablebetawith  h^{-\varepsilon\stablebetaminusonewith /q(1-2\varepsilon)}\right)\ \le\  h^\theta
  \end{align*}
  for all $h\le h_0$, possibly by decreasing $h_0$ further to accommodate the last inequality. 
\end{proof}

Now fix $r_0>0$, $\varepsilon>0$ and $\theta>0$. Following Khoshnevisan \cite{Kho-94}, we define for any $\mu>0$, $R>0$, $h>0$ and $y\in\bR$ the lattice 
$K(\mu,R)=(-R,R)\cap\mu\bZ$ and the events $G(y,h)$ given by
$$\left\{\left|\frac{h^{\stablebetaminusonewith /q}}{c}\widehat{m}_h(y,\ell^y(\lambda)-)-\ell^y(\lambda)\right|\ge \sqrt{\frac{2(\theta+\varepsilon)(1+2\varepsilon)}{c}}h^{\stablebetaminusonewith /2q}\sqrt{\log(1/h)}\sqrt{\ell^y(\lambda)},\ell^y(\lambda)\ge r_0\right\}.$$

\begin{corollary}\label{cor12} For any $w>0$ and $r>0$, let $\rho(\discretisationn)=\discretisationn^{-w}$ and $\mu(\discretisationn)=(\rho(\discretisationn))^r=\discretisationn^{-wr}$. Then
  $$\limsup_{\discretisationn\rightarrow\infty}\sup_{y\in K(\mu(\discretisationn),R)}\frac{\left|\frac{(\rho(\discretisationn))^{\stablebetaminusonewith /q}}{c}m_{\rho(\discretisationn)}(y,\lambda)-\ell^y(\lambda)\right|}{(\rho(\discretisationn))^{\stablebetaminusonewith /2q}\sqrt{\log(1/\rho(\discretisationn))}}1_{\{\ell^y(\lambda)\ge r_0\}}\le\sqrt{\frac{2(r+1/w)}{c}}\sqrt{\ell^*(\lambda)}\quad\mbox{a.s.},$$
  where $\ell^*(\lambda)=\sup_{y\in\bR}\ell^y(\lambda)$.
\end{corollary}
\begin{proof} First note by (\ref{mmhat}) and the triangular inequality, that it is equivalent to prove the statement with $m_{\rho(k)}(y,\lambda)$ replaced by
  $\widehat{m}_{\rho(k)}(y,\ell^y(\lambda)-)$, since $2(\rho(k))^{\alpha/q}/c(\rho(k))^{\alpha/2q}\sqrt{\log(1/\rho(k))}\rightarrow 0$.
  Then this proof becomes a simple application of the Borel-Cantelli lemma: for all $\theta>r+1/w$
  $$\sum_{\discretisationn\ge 1}\bP\left(\bigcup_{y\in K(\mu(\discretisationn),R)}G(y,\rho(\discretisationn))\right)\le 2R\sum_{\discretisationn\ge 1}\frac{(\rho(\discretisationn))^\theta}{\mu(\discretisationn)}<\infty$$
  implies that a.s.
  $$\limsup_{\discretisationn\rightarrow\infty}\sup_{y\in K(\mu(\discretisationn),R)}\frac{\left|\frac{(\rho(\discretisationn))^{\stablebetaminusonewith /q}}{c}m_{\rho(\discretisationn)}(y,\lambda)-\ell^y(\lambda)\right|}{(\rho(\discretisationn))^{\stablebetaminusonewith /2q}\sqrt{\log(1/\rho(\discretisationn))}}1_{\{\ell^y(\lambda)\ge r_0\}}\le\sqrt{\frac{2(\theta+\varepsilon)(1+2\varepsilon)}{c}}\sqrt{\ell^*(\lambda)},$$
  and considering countable sequences $\varepsilon_n\downarrow 0$ and $\theta_n\downarrow r+1/w$, we obtain the bound given.
\end{proof}

We seek to strengthen the corollary by replacing the supremum over $K(\mu(\discretisationn),R)$ by a supremum over $[-R,R]$:

\begin{proposition}\label{prop11} For all $R>0$ and $r_0>0$, 
we have 
  $$\limsup_{h\downarrow 0}\sup_{y\in[-R,R]}\frac{\left|\frac{h^{\stablebetaminusonewith /q}}{c}m_{h}(y,\lambda)-\ell^y(\lambda)\right|}{h^{\stablebetaminusonewith /2q}\sqrt{\log(1/h)}}1_{\{\ell^y(\lambda)\ge 2r_0\}}\le\sqrt{\frac{2q+\stablebetaminusonewith (\pathholder+1)}{q\pathholder c}}\sqrt{\ell^*(\lambda)}\qquad\mbox{a.s.}.$$
\end{proposition}
\begin{proof} Let $D_{[0,\lambda]}$ be as in Proposition \ref{propuniformholder}, for $T=\lambda$. We will work on the event of probability 1, where
  $D_{[0,\lambda]}<\infty$, $y\mapsto\ell^y(\lambda)$ satisfies Boylan's \cite{Boy1964} modulus of continuity
  $|\ell^y(\lambda)-\ell^x(\lambda)|\le K_\lambda|\log|y-x||(|y-x|)^{\stablebetaminusonewith /2}$ and where the bound of the preceding corollary holds. There is $\discretisationn_0\ge 2$ for which 
  the following two estimates hold. First, we can guarantee for all $\discretisationn\ge \discretisationn_0$ that  
  $$D_{[0,\lambda]}(\mu(\discretisationn-1))^\pathholder\le\rho(\discretisationn-1)-\rho(\discretisationn)\qquad\mbox{and}\qquad D_{[0,\lambda]}(\mu(\discretisationn+1))^\pathholder\le\rho(\discretisationn)-\rho(\discretisationn+1),$$
  provided that $r>(1+1/w)/\pathholder$. Second, we can guarantee that $|\ell^y(\lambda)-\ell^{z}(\lambda)|\le r_0$ for all $y,z\in[-R,R]$ with $|y-z|\le\mu(\discretisationn_0)$. For
  $y\in[-R,R]$ and $\discretisationn\ge 1$, let $g_\discretisationn(y)\in K(\mu(\discretisationn),R)$ with $|g_\discretisationn(y)-y|\le\mu(\discretisationn)$. Then for all $\discretisationn\ge \discretisationn_0$, we have on the chosen event of probability 1
  \begin{align*}\ell^y(\lambda)-\frac{(\rho(\discretisationn))^{\stablebetaminusonewith /q}}{c}m_{\rho(\discretisationn)}(y,\lambda) 
    &\le\left(\frac{\rho(\discretisationn)}{\rho(\discretisationn\!-\!1)}\right)^{\stablebetaminusonewith /q}\left(\!\ell^{g_{\discretisationn-1}(y)}(\lambda)-\frac{(\rho(\discretisationn\!-\!1))^{\stablebetaminusonewith /q}}{c}m_{\rho(\discretisationn-1)}(g_{\discretisationn-1}(y),\lambda)\!\right)\\
      &\qquad+\left|\ell^y(\lambda)-\ell^{g_{\discretisationn-1}(y)}(\lambda)\right|+\left(1-\left(\frac{\rho(\discretisationn)}{\rho(\discretisationn-1)}\right)^{\stablebetaminusonewith /q}\right)\ell^*(\lambda).
  \end{align*}
  We add denominators $(\rho(\discretisationn))^{\stablebetaminusonewith /2q}\sqrt{\log(1/\rho(\discretisationn))}$ and take suprema $y\in[-R,R]$ on both sides. Adding 
  indicators $1_{\{\ell^y(\lambda)\ge 2r_0\}}$ on the LHS, we may add indicators $1\{\ell^{g_{\discretisationn-1}(y)}(\lambda)\ge r_0\}$ on the first term of the RHS. 
  With these denominators, suprema and indicators, the first term of the RHS has $\limsup$ as stated in the preceding corollary, for all 
  $\discretisationn\ge \discretisationn_0$, since $\rho(\discretisationn)/\rho(\discretisationn-1)\rightarrow 1$, The other two terms vanish in the limit, 
  since $r>1/\pathholder>1/q$, provided that $w\le 2q/\stablebetaminusonewith $, which implies $(1-(\rho(\discretisationn)/\rho(\discretisationn-1))^{\stablebetaminusonewith /q})/(\rho(\discretisationn)^{\stablebetaminusonewith /2q}\log(1/\rho(\discretisationn)))\rightarrow 0$. 
  Similarly the suprema $y\in[-R,R]$ of the following bounds have bounded limsup when adding the analogous indicators and denominators:
  \begin{align*}\frac{(\rho(\discretisationn))^{\stablebetaminusonewith /q}}{c}m_{\rho(\discretisationn)}(y,\lambda)-\ell^y(\lambda)
    &\le\left(\frac{\rho(\discretisationn)}{\rho(\discretisationn\!+\!1)}\right)^{\stablebetaminusonewith /q}\left(\frac{(\rho(\discretisationn\!+\!1))^{\stablebetaminusonewith /q}}{c}m_{\rho(\discretisationn+1)}(g_{\discretisationn+1}(y),\lambda)-\ell^{g_{\discretisationn+1}(y)}(\lambda)\!\right)\\ 
    &\qquad+\left|\ell^y(\lambda)-\ell^{g_{\discretisationn+1}(y)}(\lambda)\right|+\left(\left(\frac{\rho(\discretisationn)}{\rho(\discretisationn+1)}\right)^{\stablebetaminusonewith /q}-1\right)\ell^*(\lambda).
  \end{align*}
  To strengthen the limit along $h=\rho(\discretisationn)\rightarrow 0$ to $h\downarrow 0$, write $\overline{h}=\min\left(\{\rho(\discretisationn),\discretisationn\ge 1\}\cap[h,\infty)\right)$ and $\underline{h}=\max\left(\{\rho(\discretisationn),\discretisationn\ge 1\}\cap(0,h]\right)$. Then we can bound above
  $$\frac{h^{\stablebetaminusonewith /q}}{c}m_h(y,\lambda)-\ell^y(\lambda)\le\left(\frac{\overline{h}}{\underline{h}}\right)^{\stablebetaminusonewith /q}\left(\frac{\underline{h}^{\stablebetaminusonewith /q}}{c}m_{\underline{h}}(y,\lambda)-\ell^y(\lambda)\right)+\left(\left(\frac{\overline{h}}{\underline{h}}\right)^{\stablebetaminusonewith /q}-1\right)\ell^*(\lambda),$$
  similarly below. A straightforward argument completes the proof since $w\le 2q/\stablebetaminusonewith $ so that $\overline{h}/\underline{h}\rightarrow 1$ can be strengthened to 
  $(\overline{h}/\underline{h})^{\stablebetaminusonewith /q}-1\!=\!o(h^{\stablebetaminusonewith /2q}\log(1/h))$ as $h\downarrow 0$, and with $w=2q/\alpha$ and $r_n\downarrow(1+\stablebetaminusonewith /2q)/\pathholder$, we
  establish the bound $(2q+\stablebetaminusonewith (\pathholder+1))/q\pathholder c$ as claimed.
\end{proof}

\subsection{Proofs of Theorem \ref{thm1} and Corollary \ref{corhat}}

\begin{proof}[Proof of Theorem \ref{thm1}] For the proof of the theorem, first note that Proposition \ref{prop11} gives a result at a single ${\rm Exp}(1)$ distributed
  random time. Now fix $R>0$ and consider for all $t>0$ and $r_0>0$ the events
  $$A(t,r_0)=\left\{\limsup_{h\downarrow 0}\sup_{y\in[-R,R]}\frac{\left|\frac{h^{\stablebetaminusonewith /q}}{c}m_{h}(y,t)-\ell^y(t)\right|}{h^{\stablebetaminusonewith/2q}\sqrt{\log(1/h)}}1_{\{\ell^y(t)\ge 2r_0\}}\le\sqrt{\frac{2q+\stablebetaminusonewith (\pathholder+1)}{q\pathholder c}}\sqrt{\ell^*(t)}\right\}.$$
  Then we have shown that $1=\bP(A(\lambda,r_0))=\int_0^\infty\bP(A(t,r_0))e^{-t}dt$, hence $\bP(A(t,r_0))=1$ for
  all $t\ge 0$ except possibly on a Lebesgue null set, but by scaling, this null set must be empty. 

  Now fix $T>t>0$ and let $\varepsilon>0$ with $t,\varepsilon\in\bQ$. With $r_0=\varepsilon/2$, we can a.s.\ find $h_0>0$ such that for all $h<h_0$ 
  $$\sup_{y\in[-R,R]}\left(\ell^y(t)-\frac{h^{\stablebetaminusonewith /q}}{c}m_h(y,t)\right)\le\varepsilon+\left(\varepsilon+\sqrt{\frac{2q+\stablebetaminusonewith (\pathholder+1)}{q\pathholder c}}\sqrt{\ell^*(t)}\right)h^{\stablebetaminusonewith /2q}\sqrt{\log\left(\frac{1}{h}\right)}.$$
  Taking $\limsup_{h\downarrow 0}$ and $\varepsilon\downarrow 0$ along a sequence $\varepsilon_n\downarrow 0$, we see that almost surely
  \begin{equation}\label{limsupvanish}\limsup_{h\downarrow 0}\sup_{y\in[-R,R]}\left(\ell^y(t)-\frac{h^{\stablebetaminusonewith /q}}{c}m_h(y,t)\right)\le 0.
  \end{equation}
  But on the event $\{-R\le\inf_{0\le t\le T}X_t\le\sup_{0\le t\le T}X_t\le R\}$, this supremum is actually $\sup_{y\in\mathbb{R}}$. Also, (\ref{limsupvanish}) holds
  for all $t\in\bQ\cap(0,T]$ a.s.. By continuity of $t\mapsto\ell^y(t)$ (uniformly in $y$), and the monotonicity of $t\mapsto m_h(y,t)$, this holds for all 
  $t\in[0,\infty)$ a.s.. Furthermore, for all $\varepsilon>0$, there is $\delta>0$ so that for all $y\in\bR$ and all $s<t$ in $[0,T]$, $|s-t|<\delta$ implies
  $\ell^y(t)-\ell^y(s)<\varepsilon$. So, 
  $$\limsup_{h\downarrow 0}\sup_{y\in[-R,R]}\sup_{s\in\delta\bZ\cap[0,T]}\left(\ell^y(s)-\frac{h^{\stablebetaminusonewith /q}}{c}m_h(y,s)\right)\le 0$$
  already entails
  $$\limsup_{h\downarrow 0}\sup_{y\in[-R,R]}\sup_{t\in[0,T]}\left(\ell^y(t)-\frac{h^{\stablebetaminusonewith /q}}{c}m_h(y,t)\right)\le\varepsilon.$$
  Taking a rational sequence $\varepsilon_n\downarrow 0$, this limsup must vanish a.s.. This argument worked for fixed $R>0$ on the event 
  $\{-R\le\inf_{0\le t\le T}X_t\le\sup_{0\le t\le T}X_t\le R\}$. Now taking the union of these events over a rational sequence $R_n\uparrow\infty$, we conclude that we have the ``upper bound''
  $$\limsup_{h\downarrow 0}\sup_{y\in\bR}\sup_{t\in[0,T]}\left(\ell^y(t)-\frac{h^{\stablebetaminusonewith /q}}{c}m_h(y,t)\right)\le 0\qquad\mbox{a.s..}$$
  
  For the ``lower bound'', we have similarly for all $R>0$, $t\in\bQ\cap[0,\infty)$, $r_0>0$
  $$\limsup_{h\downarrow 0}\sup_{y\in[-R,R]\colon\ell^y(t)\ge 2r_0}\left(\frac{h^{\stablebetaminusonewith /q}}{c}m_h(y,t)-\ell^y(t)\right)\le 0\qquad\mbox{a.s..}$$
  Now let $\varepsilon>0$ and choose $r_0=\varepsilon/4$. Let $S>0$. On the event
  $$\left\{-R\le\inf_{0\le t\le T}X_t\le\sup_{0\le t\le T}X_t\le R\right\}\cap\left\{\ell^y(T+S)\ge\varepsilon/2\mbox{ for all }y\in[-R,R]\right\},$$
  we find $h_0$ such that for all $h<h_0$
  $$\sup_{y\in\mathbb{R}}\,\sup_{t\in\delta\bZ\cap[0,T+S]\colon\ell^y(t)\ge 2r_0}\left(\frac{h^{\stablebetaminusonewith /q}}{c}m_h(y,t)-\ell^y(t)\right)\le\varepsilon\qquad\mbox{a.s.,}$$
  where we choose $\delta>0$ so small that $|t-s|<\delta$ in $[0,T+S]$ implies that $|\ell^y(t)-\ell^y(s)|<\varepsilon$. 
  We define $\tau^y(\varepsilon)=\inf\{t\ge 0\colon\ell^y(t)\ge\varepsilon\}$. Then we conclude that
  $$\limsup_{h\downarrow 0}\sup_{y\in\mathbb{R}}\,\sup_{t\in[0,T]}\left(\frac{h^{\stablebetaminusonewith /q}}{c}m_h(y,t)-\ell^y(t)\right)\le 3\varepsilon$$
  since for any $y\in\mathbb{R}$ and $t\in[0,T]$ with $\ell^y(t)<2r_0$ and $s$ the next lattice point in $\delta\bZ$ after $\tau^y(\varepsilon)$, we can estimate for
  $h\le h_0$  
  $$\frac{h^{\stablebetaminusonewith /q}}{c}m_h(y,t)-\ell^y(t)\le\frac{h^{\stablebetaminusonewith /q}}{c}m_h(y,\tau^y(\varepsilon/3)
    \le\frac{h^{\stablebetaminusonewith /q}}{c}m_h(y,\tau^y(\varepsilon)-\ell^y(s)+2\varepsilon\le 3\varepsilon\qquad\mbox{a.s..}$$
  For sequences $\varepsilon_n\downarrow 0$ and $R_n\uparrow\infty$, this completes the proof, as for the ``upper bound''.
\end{proof}

\begin{proof}[Proof of Corollary \ref{corhat}] This follows straight from Theorem \ref{thm1}. Just note that by symmetry of $z$
  $$\frac{1}{c^\circ}=c=\stablebetawith(\Gamma(1-\stablebetaminusonespecial))^{-1}2\int_0^{1/2}s^{\stablebetaminusonewithout}ds
               =2\frac{1}{\Gamma(1-\stablebetaminusonespecial)}\frac{1}{2^{\stablebetawithout}}
               =\frac{1}{2^{\stablebetaminusonewithout}\Gamma(1-\stablebetaminusonespecial)}.\vspace{-0.4cm}$$
\end{proof}

\section{Uniform H\"older continuity of the paths in $\kappa_q$-marked stable processes}\label{secheight}

To prove Proposition \ref{proppiling} and Corollary \ref{corcmj}, we will use a local time approximation based directly on jump height, which we present first.

\subsection{Local time approximations based on jump heights}

Consider the count $m_h^\prime(y,t)=\#\{0\le t\le T\colon X_{t-}<y<X_t,\Delta X_t>h\}$. As in 
Proposition \ref{propcmjstable}, it follows from Proposition \ref{propunif} that $\widehat{m}_h^\prime(y,r)= m_h^\prime(y,\tau^y(r))-m_h^\prime(y,\tau^y(0))$ is a 
Poisson process with rate $c^\prime h^{-\alpha}$ where $c'=  (1+\alpha)\alpha/\Gamma(1-\alpha)$, and the Poisson process arguments in the proof of Lemma \ref{lm11}, as 
well as the Borel-Cantelli argument for Corollary \ref{cor12} apply again to give for an independent $\lambda\sim{\rm Exp}(1)$:
\begin{lemma}\label{lm11prime} For all $r_0>0$, $\theta>0$, $\varepsilon>0$, there is $h_0>0$ such that for all $h\le h_0$, $y\in\bR$,
  $$\bP\left(\left|\frac{h^\alpha}{c^\prime}\widehat{m}^\prime_h(y,\ell^y(\lambda)-)-\ell^y(\lambda)\right|\ge \sqrt{\frac{2(\theta+\varepsilon)(1+2\varepsilon)}{c^\prime}}h^{\alpha/2}\sqrt{\log\left(\frac{1}{h}\right)}\sqrt{\ell^y(\lambda)},\ell^y(\lambda)\ge r_0\right)\le h^\theta.$$
\end{lemma}
\begin{corollary}\label{cor12prime} For any $w>0$ and $r>0$, let $\rho(k)=k^{-w}$ and $\mu(k)=(\rho(k))^r=k^{-wr}$. Then
  $$\limsup_{k\rightarrow\infty}\sup_{y\in K(\mu(k),R)}\frac{\left|\frac{\rho(k)^\alpha}{c^\prime}m^\prime_{\rho(k)}(y,\lambda)-\ell^y(\lambda)\right|}{(\rho(k))^{\alpha/2}\sqrt{\log(1/\rho(k))}}1_{\{\ell^y(\lambda)\ge r_0\}}\le\sqrt{2(r+1/w)}{c^\prime}\sqrt{\ell^*(\lambda)}\qquad\mbox{a.s.},$$
  where $\ell^*(\lambda)=\sup_{y\in\bR}\ell^y(\lambda)$.
\end{corollary}
The proof of Proposition \ref{prop11}, however, exploits uniform H\"older bounds on $Z_t$, $t\in[0,T]$, to show that for all $k$
large enough, we have
$$m_{\rho(k-1)}(g_{k-1}(y),\lambda)\le m_{\rho(k)}(y,\lambda)\le m_{\rho(k+1)}(g_{k+1}(y),\lambda),$$
where $g_k(y)$ is the nearest lattice point to $y$ in $K(\mu(k),R)$. This argument would need substantial change since a substitute for the H\"older bounds would need to be found to control the
number of jumps greater than $\rho(k)$ that cross level $y$. It is, however,
straightforward to find much weaker upper bounds, such as the following, which will be enough to prove Proposition \ref{proppiling}. 
\begin{proposition}\label{propheight} We have almost surely for all $T\ge 0$
  \begin{equation}\label{eq:unif_lt_cnvgc_lifetime}\limsup_{h\downarrow 0}\sup_{y\in\bR}h^\alpha m^\prime_h(y,T)<\infty.
  \end{equation}
\end{proposition}
\begin{proof} Let $\varepsilon>0$. We first claim $\limsup_{h\downarrow 0}\sup_{y\in\bR}h^\alpha m^\prime_h(y,\lambda)1_{\{\ell^y(\lambda)\ge 2r_0\}}\le 2c^\prime\ell^*(\lambda)$ a.s..
  By Corollary \ref{cor12prime}, there is $k_1\ge 1$ such that for all $k\ge k_1$ and all
  $y\in K(\mu(k),R+1)$ with $\ell^y(\lambda)\ge r_0$
  \begin{eqnarray*}\rho(k)^\alpha\widehat{m}^\prime_{\rho(k)}(y,\lambda)&\le& c^\prime\ell^y(\lambda)+c^\prime(\rho(k))^{\alpha/2}\sqrt{\log(1/\rho(k))}\left(\sqrt{2(r+1/w)}c^\prime\sqrt{\ell^*(\lambda)}+\varepsilon\right)\\ 
    &\le& c^\prime\ell^*(\lambda)+c^\prime(\rho(k))^{\alpha/2}\sqrt{\log(1/\rho(k))}\left(\sqrt{2(r+1/w)}c^\prime\sqrt{\ell^*(\lambda)}+\varepsilon\right).
  \end{eqnarray*}
  Let $r>1$. Then we have $\mu(k)<\rho(k)$. For $y\in[-R,R]$, let 
  $g_k^+(y)=\inf K(\mu(k),R+1)\cap[y,\infty)$ and $g_k^-(y)=\sup K(\mu(k),R+1)\cap(-\infty,y]$ be the lattice points of
  $K(\mu(k),R+1)$ nearest to $y$. As in the proof of Proposition \ref{prop11}, we find $k_0\ge k_1$ such that
  $|\ell^y(\lambda)-\ell^{z}(\lambda)|\le r_0$ for all $y,z\in[-R,R]$ with $|y-z|\le\mu(k_0)$. Then
  for all $y\in[-R,R]$ with $\ell^y(\lambda)\ge 2r_0$, we have $\ell^{g_k^+(y)}(\lambda)\ge r_0$ and $\ell^{g_k^-(y)}(\lambda)\ge r_0$.
  
  For all $0\le t\le T$ with $X_{t-}<y<X_t$ and 
  $\Delta X_t>\rho(k)>\mu(k)$, we must have $X_{t-}<g_k^+(y)<X_t$
  or $X_{t-}<g_k^-(y)<X_t$. Hence, 
  $$m^\prime_{\rho(k)}(y,\lambda)
    \le m^\prime_{\rho(k)}(g_k^+(y),\lambda)+m^\prime_{\rho(k)}(g_k^-(y),\lambda),$$
  so that for all $y\in[-R,R]$ with $\ell^y(\lambda)\ge 2r_0$, we have 
  $$\rho(k)^\alpha m^\prime_{\rho(k)}(y,\lambda)\le 2c^\prime\ell^*(\lambda)+2c^\prime(\rho(k))^{\alpha/2}\sqrt{\log(1/\rho(k))}\left(\sqrt{2(r+1/w)}c^\prime\sqrt{\ell^*(\lambda)}+\varepsilon\right).$$
  To get from $h=\rho(k)\rightarrow 0$ to $h\downarrow 0$, recall notation 
  $\overline{h}=\min\left(\{\rho(k),k\ge 1\}\cap[h,\infty)\right)$ and 
  $\underline{h}=\max\left(\{\rho(k),k\ge 1\}\cap(0,h]\right)$. Then we can bound above
  $$ h^\alpha m^\prime_h(y,\lambda)\le\left(\frac{\overline{h}}{\underline{h}}\right)^\alpha\underline{h}^\alpha m^\prime_{\underline{h}}(y,\lambda)$$
  and hence conclude by letting $h\downarrow 0$ to see the upper bound independent of $y$ tend to 
  $2c^\prime\ell^*(\lambda)$ to find
  \begin{equation}\label{eq:unif_lt_cnvgc_lifetime_wk}\limsup_{h\downarrow 0}\sup_{y\in[-R,R]}h^\alpha m^\prime_h(y,\lambda)1_{\{\ell^y(\lambda)\ge 2r_0\}}\le 2c^\prime\ell^*(\lambda)\qquad\mbox{almost surely.}
  \end{equation}
  Since $\cR=(X_t,0\le t\le\lambda)$ is bounded almost surely, the claim hence holds on the events 
  $\{\cR\subset[-R,R]\}$ whose union over $R\in\bN$ has probability 1, so it remains to remove the indicator. 

  Now let $T\ge 0$ and define the post-$T$ process $\widetilde{X}_t=X_{T+t}-X_T$, with local times $\widetilde{\ell}^y(t)=\ell^{y-X_T}(T+t)-\ell^{y-X_T}(T)$, $y\in\bR$,
  $t\ge 0$. Note that $\widetilde{X}$ is independent of $(X_t,0\le t\le T)$ with the same distribution as $X$. For $R\in\bN$ let 
  $$E_R=\left\{\widetilde{\ell}^y(\lambda) > 2r_0\ \mbox{for all }y\in[-R,R]\right\}.$$
  If we can prove that each of these events $E_R$ has positive probability, we obtain via \eqref{eq:unif_lt_cnvgc_lifetime_wk} that
  $$\limsup_{h\downarrow 0}\sup_{y\in [X_T-R,X_T+R]} h^\alpha m^\prime_h(y,T)\leq
    \limsup_{h\downarrow 0}\sup_{y\in [X_T-R,X_T+R]} h^\alpha m^\prime_h(y,T+\lambda)< \infty$$
  almost surely on $E_R$. But $E_R$ is independent of $(X_t,0\le t\le T)$; thus, the term on the LHS is a.s.\ finite for every $R$. Let 
  $M := S_T-I_T=\sup\{X_s,s\le t\}-\inf\{X_s,s\le t\}$. This is a.s.\ finite. In the event $\{M<R\}$, the inequality above implies \eqref{eq:unif_lt_cnvgc_lifetime}. Since 
  this event happens for some $R\in\bN$, \eqref{eq:unif_lt_cnvgc_lifetime} holds almost surely. It remains to show that $E_R$ has positive probability, which we do
  restate and prove in the following lemma. 
\end{proof}

\begin{lemma} Let $Z\sim{\rm Exp}(1)$. Then for all $\varepsilon>0$ and all $k>0$, the event $E_R$ of the proof of Proposition \ref{propheight} has
  positive probability. 
  In particular, with probability 1, $\ell^y(t)\rightarrow\infty$ as $t\rightarrow\infty$ for all $y\in\bR$, uniformly on compact $t$-intervals.
\end{lemma}
\begin{proof} Since $\tau^0$ is a stable subordinator, $\ell^0(\lambda)>0$ a.s., hence there is $\varepsilon>0$ such that $\ell^0(\lambda)>2\varepsilon$ with positive 
  probability. But then by continuity of $y\mapsto\ell^y(\lambda)$, there is $R>0$ such that with positive probability, we have  $\ell^y(\lambda)>\varepsilon$ for all 
  $y\in[-R,R]$.

  Now fix $R>0$ and consider $\tau^0(j)$, $j\ge 1$. Then the random variables $\min\{\ell^y(\tau^0(j))-\ell^y(\tau^0(j-1)),y\in[-R,R]\}$, $j\ge 1$, are independent and
  identically distributed, nonnegative and positive with positive probability. By the Strong Law of Large Numbers, their series is infinite a.s., and so
  $\ell^y(\infty)=\infty$ for all $y\in[-R,R]$, a.s.. By scaling, this holds for all $R>0$, and by choosing a sequence $R_i\rightarrow\infty$, $i\rightarrow\infty$, this
  extends to all $y\in\bR$, as required.  
\end{proof}

\subsection{Proofs of Proposition \ref{proppiling} and Corollary \ref{corcmj}}

\begin{proof}[Proof of Proposition \ref{proppiling}]
Let $J_i$ be the time of the $i$th largest jump of $(X_t, 0\leq t\leq T)$ and recall that $Z_t$ is the excursion marking the jump of $X$ at time $t$.
 We build our partition sequentially and refer to the resulting sequences $(J^k_j,\ j\geq 1)$ as ``piles'', $k\ge 1$.
 
 \emph{Piling procedure}. Place $J_1$ at the bottom of the first pile, $J^1_1 := J_1$. Now suppose that the first $i$ jump times are arranged into $k_i$ piles of heights $j_1,\ldots,j_{k_i}$, with $\sum_{k\leq k_i}j_k = i$.
 \begin{enumerate}
  \item  If for each $k\leq k_i$ there exists $j\leq j_k$ such that 
  \[[X(J_{i+1}-),X(J_{i+1})] \cap [X(J^k_j-),X(J^k_{j})] \neq \emptyset\]
  then we place it $J_{i+1}$ into a new pile by setting $J^{k_i+1}_1 := J_{i+1}$.
  \item Otherwise, we place $J_{i+1}$ atop the pile of least index $k$ for which 
  \[[X(J_{j+1}-),X(J_{j+1})] \cap [X(J^k_{j}-),X(J^k_{j})] = \emptyset\]
for every $j\leq j_k$.  I.e.\ $J^k_{j_k+1} := J_{i+1}$, where
  \begin{align*}
   k = \min\{m\leq k_i\colon [X(J_{i+1}-),X(J_{i+1})] \cap [X(J^m_j-),X(J^m_j)] = \emptyset\mbox{ for all }j\in[j_m] \}.
  \end{align*}
 \end{enumerate}
 We denote the $\theta$-H\"older constants by
 \begin{align*}
  D^k_j := \sup_{0\le a<b\le\Delta X(J_j^k)}\frac{|Z_{J^k_j}(b)-Z_{J^k_j}(a)|}{(b-a)^{\theta}},\quad j\ge 1, \quad \text{and} \quad D^k_{[0,T]} := \sup_{j\ge 1} D^k_j, \qquad \text{for all }k\geq 1.
 \end{align*}
 By definition of our piling procedure, $J^k_1$ is the time of the largest jump in the $k$th pile, for each $k\ge 1$. Consider the start and end levels $a_k=X(J^k_1-)$ and $b_k=X(J^k_1)$ of the jump at time $J^k_1$. Again, by definition of the procedure, for each $m<k$ there is some jump time $J^m_j$ that is the time of a larger jump than $J^k_1$ and such that the jump intervals intersect 
 \[[X(J^k_{1}-),X(J^k_{1})] \cap [X(J^m_{j}-),X(J^m_{j})] \neq \emptyset.\] 
Since $\Delta X(J^k_1)< \Delta X(J^m_j)$, for each such $J^m_j$, the jump interval contains one of the endpoints of the jump at time $J_k^1$:
 $$\{X(J^k_{1}-),X(J^k_{1})\} \cap [X(J^m_{j}-),X(J^m_{j})] \neq \emptyset.$$
By the pigeonhole principle, and the fact that almost surely there is no level at which more than one jump starts or ends, there is some $y_k\in (X(J^k_{1}-),X(J^k_{1}))$ such that at least $\lfloor \frac{k}{2}\rfloor$ jumps larger than $\Delta X_{J^k_1}$ jump across level $y_k$.  That is,
\[ \#\left\{ i\ge 1\colon \Delta X(J_i) > \Delta X(J^k_1) \textrm{ and } y_k\in [X(J_{i}-),X(J_{i})] \right\} \geq \left\lfloor \frac{k}{2}\right\rfloor.\]
 
 Proposition \ref{propheight} implies that there is almost surely some $C \in (0,\infty)$ such that for every $y\in\bR$ and every $k\geq 1$, the $k$th  largest jump across level $y$ has size at most $Ck^{-1/\alpha}$.  Since $\Delta X(J^k_1)$ is at most the $\lfloor \frac{k}{2}\rfloor$th largest jump across level $y_k$, we see that $\Delta X(J^k_1) \leq C\lfloor\frac{k}{2}\rfloor^{-1/\alpha} \leq C'k^{-1/\alpha}$.  Thus $\Delta X(J^k_j) \leq C'k^{-1/\alpha}$ for all $j\geq 1$.
 
 Take $\varepsilon > 0$. Then, from the above, there is a.s.\ some finite $K$ such that for $k>K$, the jump size $\Delta X(J^k_1)$ is at most $k^{\varepsilon-1/\alpha}$. Let
 \begin{equation}
  D^* := \sup_{k\geq 1,\ j\geq1}D^k_j \quad \text{and} \quad D' := \sum_{m=1}^{\infty} \sup\left\{D^k_j\colon \Delta X(J^k_1)< m^{\varepsilon-1/\alpha},\ k\geq 1,\ j\geq1\right\},
\end{equation}
so that
\[
  \sum_n D_n \leq K D^* + D'.
\]

 From Proposition \ref{propuniformholder}, $D^*$ is a.s.\ finite, and $K$ is a.s.\ finite as well, so it suffices to show that $D'$ is a.s.\ finite. We appeal to \eqref{eq:all_short_excs_holder} to find that
 \begin{align}
  \bE(D') \leq \sum_{m=1}^\infty C_{\varepsilon}m^{(\varepsilon-\frac{1}{\alpha})(q - \theta - \varepsilon)}.
 \end{align}
 Since $\theta \leq \gamma$ and $q>\gamma+\alpha$ this series converges for all sufficiently small $\varepsilon$.
\end{proof}

\begin{proof}[Proof of Corollary \ref{corcmj}] For this proof, abbreviate $\markZ_t^y:=\markZ_t(y\!-\!X_{t-})$ and $\markZ_t^x:=\markZ_t(x\!-\!X_{t-})$. Then
  \begin{align*}|\markZ_{[0,T]}(y)-\markZ_{[0,T]}(x)|&=\left|\sum_{0\le t\le T}(\markZ_t^y-\markZ_t^x)\right|\le\sum_{0\le t\le T}\left|\markZ_t^y-\markZ_t^x\right|\\
                                                      &\le\sum_{0\le t\le T\colon \markZ_t^y\neq 0}|\markZ_t^y-\markZ_t^x|
                                                        +\sum_{0\le t\le T\colon \markZ_t^x\neq 0}|\markZ_t^y-\markZ_t^x|
                                                      \le\left(2\sum_{n\ge 1}D_n\right)|y-x|^\pathholder
              \end{align*}      
      by Proposition \ref{proppiling}.        
\end{proof}

\section{Moments of H\"older constants of BESQ processes, bridges and excursions}\label{sect2}

\subsection{Brownian motion, Brownian bridge and Brownian excursion}

\begin{lemma}\label{lm1} Let $(B_t,0\le t\le 1)$ be standard Brownian motion and $\pathholder\in(0,1/2)$. Then
  $$D_\pathholder=\sup_{0\le s<t\le 1}\frac{|B_t-B_s|}{|t-s|^\pathholder}<\infty\qquad\mbox{a.s.}$$
  and the uniform bound $D_\pathholder$ has moments of all orders.
\end{lemma}
\begin{proof} This is well-known and follows straight from the Kolmogorov-Chentsov theorem, see e.g. Revuz and Yor \cite[Theorem I.(2.1)]{RY}. Specifically, by scaling, we have $\bE(|B_t-B_s|^{2p})=C_p|t-s|^p$ for all $p>0$, so the theorem yields
$\bE(D_\pathholder^p)<\infty$ as long as $p>1$ and $0<\pathholder<1/2-1/(2p)$, i.e. $0<\pathholder<1/2$ and $p>1/(1-2\pathholder)$. This gives moments of all orders, as required.
\end{proof}

\begin{corollary} Let $(B^{\rm br}_t,0\le t\le 1)$ and $(B^{\rm ex}_t,0\le t\le 1)$ be standard Brownian bridge and standard Brownian excursion and $\pathholder\in(0,1/2)$. Then
$$D_\pathholder^{\rm br}=\sup_{0\le s<t\le 1}\frac{|B_t^{\rm br}-B^{\rm br}_s|}{|t-s|^\pathholder}<\infty\quad\mbox{and}\quad D_\pathholder^{\rm ex}=\sup_{0\le s<t\le 1}\frac{|B_t^{\rm ex}-B_s^{\rm ex}|}{|t-s|^\pathholder}<\infty\qquad\mbox{a.s.}$$
and the uniform bounds $D_\pathholder^{\rm br}$ and $D_\pathholder^{\rm ex}$ have moments of all orders.
\end{corollary}
\begin{proof} We use the pathwise representations, due to L\'evy and Vervaat: 
  $$B_t^{\rm br}=B_t-tB_1,\ 0\le t\le 1,\quad{and}\quad B_t^{\rm ex}=\left\{\begin{array}{ll}B_{M+t}^{\rm br},&0\le t\le 1-M,\\ B_{t-(1-M)},&1-M\le t\le 1,\end{array}\right.$$
  where $M=\inf\{t\ge 0\colon B_t^{\rm br}=\min\{B_r^{\rm br},0\le r\le 1\}\}$. Then we have for all $0\le s<t\le 1$ 
  $$|B_t^{\rm br}-B_s^{\rm br}|\le|B_t-B_s|+|t-s||B_1-B_0|\le 2D_\pathholder|t-s|^\pathholder$$
  and similarly for $|B_t^{\rm ex}-B_s^{\rm ex}|$ if $0\le s<t\le 1-M$ or $1-M\le s<t\le 1$. For $0\le s<1-M<t\le 1$
  $$|B_t^{\rm ex}-B_s^{\rm ex}|=|B_{M+s}^{\rm br}-B_{t-(1-M)}^{\rm br}|\le|B_{M+s}^{\rm br}-B_1^{\rm br}|+|B_{t-(1-M)}^{\rm br}-B_0^{\rm br}|\le 4D_\pathholder|t-s|^\pathholder.$$
  Hence $D_\pathholder^{\rm br}\le 2D_\pathholder$ and $D_\pathholder^{\rm ex}\le 4D_\pathholder$.
\end{proof}

Via $|(B_t^{\rm br})^2-(B_s^{\rm br})^2|=|B_t^{\rm br}-B_s^{\rm br}||B_t^{\rm br}+B_s^{\rm br}|\le 4D_\pathholder^{\rm br}\sup_{0\le r\le 1}|B_r^{\rm br}|$, 
Cauchy-Schwarz (or binomial formulas), the fact that $\sup_{0\le r\le 1}|B_r^{\rm br}|\le\sup_{0\le r\le 1}|B_r^{\rm ex}|$ in the Vervaat coupling, and the fact that
the maximum of $B^{\rm ex}$ has moments of all orders, these results extends to ${\rm BESQ}(1)$-bridges. By \cite[Remark (5.8)(i)]{PiY2}, on the one hand sums of 
${\rm BESQ}(1)$ bridges are ${\rm BESQ}(\delta)$ bridges for $\delta\in\bN$, while suitable couplings of non-integer dimensions allow to extend this to any real dimension
$\delta\ge 1$, and on the other hand, these bridges are also the normalised ${\rm BESQ}(4-\delta)$ excursions, so all these processes have H\"older constants with moments
of all orders. 

We will use a different approach to also include results for ${\rm BESQ}$ processes starting from $x\neq 0$. Intuitively, such processes are of interest in the context of
the present paper, and further work in \cite{Paper1}, as the processes $(Z_t(y-X_{t-}),y\ge 0)$ in the study of the process 
$Z_{[0,T]}(y)=\sum_{0\le t\le T}Z_t(y-X_{t-})$, $y\ge 0$ of (\ref{cmj}).

\subsection{Results for general diffusions of BESQ-type}

Our arguments are adapted from those by Hutzenthaler et al. \cite{Hutz}, where similar results are obtained for a different class of processes.

\begin{lemma}
\label{l:reg_CIR}
Let $(X_t,t\ge 0)$ be a ${\rm BESQ}(\delta)$ process starting from $x\ge 0$, for some dimension parameter $\delta\ge 0$. Then
$$\left\|X_t\right\|_p\le x+t(\delta+2(p-1)^+)\qquad\mbox{for all }t\ge 0\mbox{ and }p\in(0,\infty).$$
\end{lemma}
\begin{proof} Recall, e.g.\ from \cite[Definition 1]{GoY2003} that we may assume that for a Brownian motion $B$ 
  \begin{equation}\label{sde}X_t=x+\delta t+2\int_0^t\sqrt{X_s}dB_s\qquad a.s.,\qquad t\ge 0.\end{equation}
By continuity of $p$-norms, it suffices to show that for all 
$ \eps \in (0,1) $ and $k\ge 0$
\begin{equation} 
 \label{indhyp} 
 \left\| X_t + \eps \right\|_{p}
\leq    x+ \eps+ t \left(c+2(p-1)^+ \right)\qquad\mbox{for all } t\ge 0\mbox{ and }p \in ( k, k + 1 ].
\end{equation}
We will prove this by induction on $k$. For $k=0$ note that by (\ref{sde})
$$\bE(X_t)=x+\delta t\quad\Rightarrow\quad(\bE((X_t+\eps)^p))^{1/p}\le\bE(X_t+\eps)=x+\eps+t\delta\quad\mbox{for all }p\in(0,1].$$
Assuming (\ref{indhyp}) for some $k\ge 0$, we obtain for all $t\ge 0$ and $p\in(k+1,k+2]$ by It\^o's lemma that
\begin{align*}
  \bE\left(\left(X_t+\eps\right)^p\right)
  &=(x+\eps)^p+\int_0^t\left(\delta p\bE\left(\left(X_s+\eps\right)^{p-1}\right)+2p(p-1)\bE\left(X_s\left(X_s+\eps\right)^{p-2}\right)\right)ds\\
  &\le(x+\eps)^p+\int_0^tp(\delta+2(p-1))\bE\left(\left(X_s+\eps\right)^{p-1}\right)ds\\
  &\le(x+\eps)^p+\int_0^tp(\delta+2(p-1))\left(x+\eps+s(\delta+2(p-2)^+)\right)^{p-1}ds\\
  &\le(x+\eps)^p+(x+\eps+t(\delta+2(p-1))^p-(x+\eps)^p=(x+\eps+t(\delta+2(p-1))^p.
\end{align*} \vspace{-0.7cm}

\end{proof}

\begin{lemma}\label{l:reg_CIR22}
  Let $\mu\colon\bR\to\bR$ be Lipschitz and bounded by $c>0$. Let $(B_t,t\ge 0)$ be Brownian motion. For $x\in\bR$, let $(X_t,t\ge 0)$ be a stochastic process with
  continuous sample paths adapted to the same filtration as $B$ and satisfying $\int_0^t|\mu(X_s)|ds<\infty$ a.s. and
  \begin{equation}\label{sde2}X_t=x+\int_0^t\mu(X_r)dr+2\int_0^t\sqrt{|X_r|}dB_r\qquad a.s.,\qquad t\ge 0.
  \end{equation}
  Then for all $t\ge s\ge 0$ and $p\in[2,\infty)$ we have that
  $$\left\|X_t-X_s\right\|_p\le\sqrt{t-s}\left(\sqrt{t}c+\sqrt{t}2\sqrt{(p-1)(c+p-2)}+2\sqrt{p-1}\sqrt{|x|+2t(c+p-2)}\right).$$
\end{lemma}

We remark that for any $\mu$ satisfying the hypotheses of the Lemma, the
weak existence and pathwise uniqueness hold for the SDE (13) by
\cite[Theorem 5.3.8 and Theorem 5.3.10]{EK86}.  Strong existence of
solutions and the Markov property of solutions then follows from
\cite[Theorem 21.11 and Theorem 21.14]{kal}

\begin{proof}
By passing to $-X$ if needed, we may assume that $x\geq 0$.  Using the strong existence of solutions to the squared Bessel process SDE and the comparison of 
one-dimensional diffusions, there exists a $\besq(c)$ process $(Z_t)_{t\geq 0}$ starting from $x$ and a $\besq(-c)$ process $(Y_t)_{t\geq 0}$ starting from 0 such that 
$Y_t\leq X_t \leq Z_t$ for all $t\ge 0$.  Since $Y_t \leq 0$ and $Z_t\geq 0$, we have that $|X_t|^p \leq Z^p_t + |Y_t|^p$.  Note that $|Y_t|$ is a $\besq(c)$ process
starting from 0.  Using the inequality $(x+y)^{1/p} \leq x^{1/p}+y^{1/p}$ for $x,y>0$ and $ p\geq 1$, Lemma \ref{l:reg_CIR} implies that
\begin{equation}\label{modx} \left\| X_t \right\|_{p}\leq \left\| Z_t \right\|_{p} + \left\| Y_t \right\|_{p}\leq |x|+ 2t(c+2(p-1)) .
\end{equation} 
Now consider the case $s=0$. Then
$$  
\left\|X_t - x   \right\|_{ 
    p }
 \leq
\left\|
\int_0^t |\mu(X_r)| \, dr 
 \right\|_{ 
    p }
+
\left\|
\int_0^t 2\sqrt{|X_r|} dB_r
 \right\|_{ 
    p }
\leq
ct
+
\left\|
\int_0^t2\sqrt{|X_r|} dB_r
 \right\|_{ 
    p }.
$$
Let $M_t = \int_0^t2\sqrt{|X_r|} dB_r$.  Since $p\geq 2$, It\^{o}'s lemma implies
\[ M_t^p = \int_0^t pM_r^{p-1}dM_r + \frac{p(p-1)}{2} \int_0^t M_r^{p-2} d[M]_r = \int_0^t pM_r^{p-1}dM_r + 2p(p-1) \int_0^t M_r^{p-2} |X_r| dr.\]
Taking expectations and using H\"older's inequality,
\[ \bE\left(M_t^p\right) = 2p(p-1) \int_0^t \bE\left(M_r^{p-2} |X_r|\right) ds \leq 2p(p-1) \int_0^t \left(\bE\left(M_r^{p}\right)\right)^{(p-2)/p} \left(\bE\left(|X_r|^{p/2}\right)\right)^{2/p} dr .\]  
Since $f(r) = \bE(M_r^p)$ is continuous, the generalized Gronwall inequality, see e.g. Bihari \cite{Bihari1956}, yields
\[\bE(M_t^p) \leq \left(4(p-1)(|x|t+t^2(c + p-2))\right)^{p/2}.\]
Consequently, 
\begin{equation}\label{initial}\left\|X_t - x   \right\|_{p } \leq ct + 2\sqrt{(p-1)\left(|x|t+t^2(c + (p-2))\right)}.
\end{equation}
Now let $t\ge s\ge 0$. By the Markov property of $X$ at time $s$ and equations (\ref{initial}) and (\ref{modx}), we get
\begin{align*} 
  &\left(\bE\left( |X_t-X_s|^p\right)\right)^{1/p}
  \le\left(\bE\left(\left(c(t-s)+2\sqrt{(p-1)\left(|X_s|(t-s)+(t-s)^2(c+p-2)\right)}\right)^p\right)\right)^{1/p}\\
  &\le\left(\bE\left(\left((t-s)\left(c+2\sqrt{(p-1)(c+p-2)}\right)+2\sqrt{p-1}\sqrt{t-s}\sqrt{|X_s|}\right)^p\right)\right)^{1/p}\\
  &\le(t-s)\left(c+2\sqrt{(p-1)(c+p-2)}\right)+2\sqrt{p-1}\sqrt{t-s}\sqrt{|x|+2s(c+2(p/2-1))}.
\end{align*}
\vspace{-0.7cm}

\end{proof}

\begin{corollary} \label{cor gen holder}
Let $(X_t,0\leq t\leq 1)$ satisfy the hypotheses of Lemma \ref{l:reg_CIR22}.  Then for every $p \in (2,\infty)$ and $\pathholder\in (0,(p-2)/2p)$,
$$ \bE\!\left(\! \left(  \sup_{0\leq s <t \leq 1}\!\! \frac{|X_t-X_s|}{|t-s|^\pathholder}\right)^{\!p} \right) \leq \frac{ 2^{\pathholder p+p+1}\left( c+2\sqrt{(p\!-\!1)(c\!+\!p\!-\!2)}+2\sqrt{p\!-\!1}\sqrt{|x|\!+\! 2(c\!+\!p\!-\!2) }\right)}{\left(1-2^{\pathholder + (2-p)/2p}\right)^p} .$$
\end{corollary}

\begin{proof}
The fact that 
\[ \bE\left(\left(  \sup_{0\leq s <t \leq 1} \frac{|X_t-X_s|}{|t-s|^\pathholder}\right)^p \right) <\infty\]
is an immediate consequence of \cite[Theorem I.(2.1)]{RY} and the explicit bound as a function of $x$, $c$, $p$ and $\gamma$ is obtained by keeping track of the 
constants in the proof of that theorem. 
\end{proof} 

\subsection{Results for BESQ processes}

The following corollary is an important special case of Corollary \ref{cor gen holder}. 

\begin{corollary}\label{cardelta} For $\delta\ge 1$, let $(X_t,0\le t\le 1)$ be a ${\rm BESQ}(\delta)$ process starting from 0, and $\pathholder\in(0,1/2)$. Then 
   $$D_{\delta,\pathholder}=\sup_{0\le s<t\le 1}\frac{|X_t-X_s|}{|t-s|^\pathholder}<\infty\qquad\mbox{a.s.}$$
   with moments of all orders.
\end{corollary}

\begin{corollary}\label{cor general bridge holder} For $\delta\ge 1$, let $Z$ be a standard ${\rm BESQ}(\delta)$ bridge from 0 to 0 or equivalently a ${\rm BESQ}(4-\delta)$ excursion, and let  $\pathholder\in(0,1/2)$. Then 
   $$D_\pathholder^*=\sup_{0\le s<t\le 1}\frac{|Z_t-Z_s|}{|t-s|^\pathholder}<\infty\qquad\mbox{a.s.}$$
   with moments of all orders.
\end{corollary} 
\begin{proof} The equivalence of ${\rm BESQ}(\delta)$ bridges and ${\rm BESQ}(4-\delta)$ excursions was noted in \cite[Remark (5.8)(i)]{PiY2}. By \cite[Exercises XI.(3.6)-(3.7)]{RY}, the process $(Z_{1-t},0\le t\le 1)$ has the same distribution as $Z$ and we can write $Z_u=u^2 X_{1/u-1}$ for a ${\rm BESQ}(\delta)$ process $X$ starting from $0$. For $1/2\le s<t\le 1$, we then obtain by the previous corollary
\begin{eqnarray*}|Z_t-Z_s|&=&|t^2X_{1/t-1}-s^2X_{1/s-1}|\ \le\ |X_{1/t-1}-X_{1/s-1}|+(t^2-s^2)X_{1/s-1}\\
  &\le& 4^\pathholder D_{\delta,\pathholder}|t-s|^\pathholder+2|t-s|^\pathholder\overline{X}.
\end{eqnarray*}
Similarly, we can write $Z_u=(1-u)^2X_{1/(1-u)-1}$ for another (dependent!) ${\rm BESQ}(\delta)$ process $X$ starting from $0$, and for $0\le s<t\le 1/2$, we obtain $|Z_t-Z_s|\le (4^\pathholder\widetilde{D}_{\delta,\pathholder}+2\overline{X})|t-s|^\pathholder$. Finally, for $0\le s<1/2<t\le t$,
the triangular inequality yields the required bound so that 
$$D_\pathholder^*\le 4^\pathholder D_{\delta,\pathholder}+2\overline{X}+4^\pathholder\widetilde{D}_{\delta,\pathholder}+2\overline{X}$$
has moments of all orders.
\end{proof}

\subsection{BESQ-marked stable processes and their local times}\label{secmarkbes}

If $X$ has Laplace exponent $\psi(\eta)=\eta^{\stablebetawithout}$, then $X_a(t)=X_{at}$, $t\ge 0$, has Laplace exponent $\psi_a(\eta)=a\eta^{\stablebetawithout}$. From the 
occupation density formula, we see easily that
$$\ell_a^y(t)=\frac{1}{a}\ell^y(at)\qquad\mbox{and}\qquad\tau_a^y(s)=\frac{1}{a}\tau^y(as).$$
Leaving marks unscaled, $m_{h,a}(y,t)=m_h(y,at)$, and Theorem \ref{thm1} holds with $c_a=ac$, Corollary \ref{corhat} with $c^\circ_a=c^\circ/a$. 

Specifically, choosing $q=1$ and $\kappa_1$-scaled ${\rm BESQ}(-2\alpha)$ excursions, we can calculate
$$\bE(Z_U^\stablebetaminusonewithout)=\frac{2^\stablebetaminusonewithout\Gamma(\stablebetawithout)}{\stablebetawithout}\quad\Rightarrow\quad 
  c_a=a\frac{\Gamma(\stablebetawithout)}{\Gamma(1-\stablebetaminusonespecial)}2^\alpha.$$
Apart from $a=1$, there are various other natural choices that keep constants in certain formulas simple. For the following, we choose $c_a=1/\Gamma(1-\alpha)$, so that
Proposition \ref{propcmjstable} yields Laplace exponent $\Theta_a(\xi)=\xi^\alpha$. This corresponds to setting $a=1/\Gamma(1+\alpha)2^\alpha$. For $\alpha=1/2$, these 
are $a=\sqrt{2/\pi}$ and $c_a=1/\sqrt{\pi}$. For illustration, let us restate Theorem \ref{thm1} in this special case. 

\begin{theorem}\label{thm1besq} Let $\stablebetaminusonespecial  \in(0,1)$ and $a=1/\Gamma(1+\alpha)2^\alpha$. For a ${\rm BESQ}(-2\alpha)$-marked stable process $X_a$ 
  with Laplace exponent $\psi_a(\eta)=\eta^{\stablebetawithout}/\Gamma(1+\alpha)2^\alpha$, we have almost surely
  $$\lim_{h\downarrow 0}\,\sup_{0\le t\le T}\,\sup_{y\in\mathbb{R}}\left|\Gamma(1-\alpha)h^{\stablebetaminusonewith}m_{h,a}(y,t)-\ell^y_a(t)\right|=0,\qquad\mbox{for all $T>0$,}$$
  where local times $\ell^y_a(t)$ and mass counts $m_{h,a}(y,t)$ are associated with $X_a$ for $y\!\in\!\bR$, $t\!\ge\!0$, $h\!>\!0$.
\end{theorem}
\begin{proof} To apply Theorem \ref{thm1}, let us first note that the H\"older constant of a ${\rm BESQ}(-2\alpha)$ excursion has moments of all orders by Corollary
  \ref{cor general bridge holder}. 
  Then, for our choice $a=1/\Gamma(1+\alpha)2^\alpha$ and $c_a=ac$, we find that $$\left|\Gamma(1-\alpha)h^{\stablebetaminusonewith}m_{h,a}(y,t)-\ell_a^y(t)\right|=\frac{1}{a}\left|\frac{h^{\stablebetaminusonewith}}{c}m_h(y,at)-\ell^y(at)\right|,$$
  and the application of Theorem \ref{thm1} yields that the relevant suprema tend to 0 as $h\downarrow 0$. 
\end{proof}

\section{Uniform H\"older continuity of local times}\label{sectrestricted}

Boylan \cite{Boy1964} established that the family of local times of $\stablebetawith $-stable L\'evy processes admits a version that is H\"older continuous of order $\localtimeholder\in(0,\stablebetaminusonewith /2)$ in the spatial direction, uniformly in space and time on any compact space-time rectangles. Barlow \cite{Bar88} gave the exact modulus of continuity. While Barlow's results are optimal in rate, we are interested in a (random) H\"older constant with all moments and which applies uniformly on space-time rectangles that consist of a fixed compact spatial interval and relevant random time intervals, whose random lengths have infinite mean. To achieve this without getting into the technicalities of Barlow's argument, we sacrifice a slowly varying function in the exact modulus of continuity and restrict the stable process to a spatial interval.

\subsection{Spectrally one-sided L\'evy processes restricted to an interval}

Consider at first any zero mean spectrally \em negative \em L\'evy process $\widehat{X}$ with unbounded variation and zero Gaussian coefficient, with notation to facilitate passage to the spectrally positive process $X=-\widehat{X}$ later. For the process $\widehat{X}$, 0 is regular for itself and for $(-\infty,0)$, by \cite[Corollary VII.5]{Ber1996}. By \cite[Theorems IV.4 and IV.10]{Ber1996}, $\widehat{X}$ admits a continuous local time at 0 and hence at any level, with an associated excursion process, whose intensity measure we denote by $\widehat{n}$.

By \cite[Proposition 1]{Ber1996}, the same holds for the reflected process $\widehat{Y}=\widehat{X}-\widehat{I}$, where $\widehat{I}_t=\inf_{0\le s\le t}\widehat{X}_s$, $t\ge 0$; as in Section \ref{secprel}, we denote the excursion measure by $\underline{\widehat{n}}$. The measures $\widehat{n}$ and $\underline{\widehat{n}}$ are measures on the Skorohod space $\bD([0,\infty),\bR)$, supported by the subspace of paths $\omega=(\omega(s))_{s\ge 0}\in\bD([0,\infty),\bR)$ for which $\omega(s)\neq 0$ if and only if  $s\in(0,T_0(\omega))$, where $T_0(\omega)=\inf\{t>0\colon\omega(t)=0\}\in(0,\infty]$ is called the lifetime of the excursion $\omega$. Since $\widehat{X}$ is spectrally negative, $\widehat{n}$-a.e. excursion $\omega$ has at most one jump across zero, whose time we denote by $T_{\le 0}(\omega)=\inf\{t>0\colon\omega(t)\le 0\}$. By \cite[Lemma 2]{PPR15arxiv}, there is exactly one jump across zero $\widehat{n}$-a.e., under our assumption of a zero Gaussian coefficient. By \cite[Theorem 3]{PPR15arxiv}, 
$$\underline{\widehat{n}}=\widehat{n}((\omega(s)1_{\{s<T_{\le 0}\}})_{s\ge 0}\in\cdot).$$
In particular, if we define the time change 
$$\tau^+(r)=\inf\{t\ge 0\colon R^+(t)>r\},\qquad\mbox{where }R^+(t)=\int_0^t1_{\{\widehat{X}_s\ge 0\}}ds,$$
so that $\widehat{X}^+_r=\widehat{X}_{\tau^+(r)}$ has all negative parts of excursions removed, then $\widehat{X}^+$ has the same distribution as $\widehat{Y}$. The analogous operation that removes all negative parts of excursions from the spectrally positive process $X=-\widehat{X}$ also yields a strong Markov process $X^+$, but this process does \em not \em have the same distribution as $Y=X-I$, where $I_t=\inf_{0\le s\le t}X_s$, because excursions of $Y$ start continuously (none of the countably many jumps strike during the Lebesgue-null set of times when $Y=0$, by independence properties of Poisson point processes), while excursions of $X^+$ away from 0 start with a jump, with intensity $\chi=n(\omega(T_{\ge 0})\in \cdot\,)=\widehat{n}(-\omega(T_{\le 0})\in\cdot\,)$, where $n=\widehat{n}(-\omega\in\cdot\,)$ is the excursion measure of $X$ at $0$ and $T_{\ge 0}(\omega)=\inf\{t>0\colon\omega(t)\ge 0\}$. By the strong Markov property under $n$ at $T_{\ge 0}$, the post-$T_{\ge 0}$ process under $n$ is that of the spectrally positive L\'evy process absorbed at 0. Cf.\ \cite[Theorem 3]{PPR15arxiv}. Cf. also \cite[Section 3]{KPW14}, where processes with negative parts of excursions removed are introduced for other stable L\'evy processes. By Vigon's \'equations amicales (see \cite[Theorem 16]{Don06}), $\chi$ is absolutely continuous with 
\begin{equation}\label{amicale}\chi(dx)=\overline{\Pi}(x)dx=\Pi((x,\infty))dx, 
\end{equation} 
where $\Pi$ is the L\'evy measure of $X$ so that 
$$\psi(\eta)=\log(\bE(\exp(-\eta X_1)))=\int_{(0,\infty)}(e^{-\eta x}-1+\eta x)\Pi(dx).$$
Finally, we define the spectrally positive L\'evy process restricted to an interval $[0,\thresholda]$ as $X^\thresholda_r=X_{\tau^{[0,\thresholda]}(r)}$, where
$$\tau^{[0,\thresholda]}(r)=\inf\{t\ge 0\colon R^{[0,\thresholda]}(t)>r\},\qquad\mbox{and }R^{[0,\thresholda]}(t)=\int_0^t1_{\{0\le X_s\le \thresholda\}}ds.$$
Note that this process is different from (actually simpler than) Pistorius' doubly reflected process of \cite{Pis03}, since here the boundary behaviour at 0 is not reflection, and is also different from Lambert's \cite{Lambert2000} process confined in a finite interval (conditioned not to exit $(0,\thresholda)$). The discussion of these three cases is connected to the discussion in \cite{CaC06} and \cite[Remark 3.3]{KPW14}, which distinguish three types of exit at a boundary (A) continuously, (B) by a jump or (C) not at all (while preserving self-similarity of the process). In our context (two boundaries, self-similarity being meaningless on an interval, but re-entry being allowed), Lambert studies (C,C) exits from $(0,\thresholda)$, no entry needed, Pistorius studies the (A,B) exit (A,A) entrance, and we find (A,B) exit (B,A) entrance. Spectral positivity disallows B exit at 0 and B entrance from $\thresholda$, but this leaves a number of other possibilities, in principle, which we do not 
pursue further.

Scale functions $W^{(q)}$ with Laplace transform $\int_0^\infty e^{-\eta x}W^{(q)}(x)dx=1/(\psi(\eta)-q)$ and $Z^{(q)}(x)=1+q\int_0^xW^{(q)}(z)dz$ are well-known
functions in the fluctuation theory of spectrally one-sided L\'evy processes, see e.g.\ \cite{Pis03}.  

\begin{proposition}\label{prop12} The Laplace exponent of the inverse local time $\sigma^\thresholda$ of $X^\thresholda$ at 0 is given by \begin{eqnarray*}\Theta^\thresholda(q)\!\!\!&=&\!\!\!-\log(\bE(\exp(-q\sigma_1^\thresholda)))=\int_0^\thresholda\left(1-\frac{Z^{(q)}(\thresholda-x)}{Z^{(q)}(\thresholda)}\right)\chi(dx)+\left(1-\frac{1}{Z^{(q)}(\thresholda)}\right)\chi((\thresholda,\infty))\\
   &=&\!\!\!\frac{1}{Z^{(q)}(\thresholda)}\int_0^\thresholda\overline{\overline{\Pi}}(\thresholda-z)qW^{(q)}(z)dz,
\end{eqnarray*}
where $\overline{\overline{\Pi}}(x)=\int_x^\infty\overline{\Pi}(z)dz$. Furthermore, $\sigma^\thresholda$ has moments of all orders.   
\end{proposition}
\begin{proof} First note that $\overline{\overline{\Pi}}(x)<\infty$ for all $x>0$ since $X$ has finite (zero) mean. After the discussion preceding the proposition, this follows directly from the exponential formula for the Poisson point process of excursions away from zero and the first exit problem of the reflected process $\widehat{X}^+$ from $(0,\thresholda)$. Specifically, $T_{\ge 0}(\omega)=0$ for $n^\thresholda$-a.e. $\omega$, and the post-$T_{\ge 0}$ process under the excursion measure $n^\thresholda$ of $X^\thresholda$ away from $0$ is that of the restricted process aborbed at 0, which is the same as $\thresholda$ minus the reflected process $\widehat{X}^+$ absorbed at $\thresholda$. By \cite[Proposition 2]{Pis04}, 
  $$\bE_x\left(e^{-qT_0(X^\thresholda)}\right)=\frac{Z^{(q)}(\thresholda-x)}{Z^{(q)}(\thresholda)},\qquad x\in(0,\thresholda],$$
For $x>\thresholda$, the post-$T_{\ge 0}$ process under $n^\thresholda$ starts from $\thresholda$. Let $W=W^{(0)}$ and recall e.g.\ from \cite{Ber97}
$$W^{(q)}(x)=\sum_{k\ge 0}q^kW^{*(k+1)}(x),\quad\mbox{where for the convolution }W^{*(k+1)}(x)\le\frac{1}{k!}x^kW^{(0)}(x)^{k+1},$$
so that $q\mapsto W^{(q)}(x)$ is analytic on $\bC$. Hence, we calculate by (\ref{amicale}) and using Fubini's Theorem
\begin{eqnarray*}\Theta^\thresholda(q)\!\!\!&=&\!\!\!-\log(\bE(\exp(-q\sigma_1^\thresholda)))=\int_0^\thresholda\left(1-\frac{Z^{(q)}(\thresholda-x)}{Z^{(q)}(\thresholda)}\right)\chi(dx)+\left(1-\frac{1}{Z^{(q)}(\thresholda)}\right)\chi((\thresholda,\infty))\\
  &=&\!\!\!\frac{1}{Z^{(q)}(\thresholda)}\left(\int_0^\thresholda\int_{\thresholda-x}^\thresholda qW^{(q)}(z)dz\overline{\Pi}(x)dx+\int_0^\thresholda qW^{(q)}(z)dz\int_\thresholda^\infty\overline{\Pi}(x)dx\right)\\
  &=&\!\!\!\frac{1}{Z^{(q)}(\thresholda)}\int_0^\thresholda\overline{\overline{\Pi}}(\thresholda-z)qW^{(q)}(z)dz\\
  &=&\!\!\!\frac{1}{Z^{(q)}(\thresholda)}\int_0^\thresholda\int_{\thresholda-z}^\infty\overline{\Pi}(x)dxq\sum_{k\ge 0}q^kW^{*(k+1)}(z)dz
  =\frac{1}{Z^{(q)}(\thresholda)}\sum_{j\ge 1}q^j\int_0^\thresholda\overline{\overline{\Pi}}(\thresholda-z)W^{*(j)}(z)dz.  
\end{eqnarray*}
Now $\Theta^\thresholda(q)$ is a ratio of complex power series, hence infinitely differentiable within the radius where the (analytic) denominator is non-zero. Since $Z^{(q)}(0)=1$, all moments of $\sigma^\thresholda$ are finite. 
\end{proof}

\begin{corollary} If $X$ is a spectrally positive stable process with Laplace exponent $\psi(\eta)=a\eta^\stablebetawithout$ for some $c>0$ and $\stablebetawithout\in(1,2)$, then 
$$\Theta^\thresholda(q)=\frac{\thresholda^{-\stablebetaminusonespecial}}{E_\stablebetawithout(q\thresholda^\stablebetawithout/a)}\left(E_{\stablebetawithout,1-\stablebetaminusonespecial}(q\thresholda^\stablebetawithout/a)-\frac{1}{\Gamma(1-\stablebetaminusonespecial)}\right),$$
where $E_{\stablebetawithout,\beta}(x)=\sum_{k\ge 0}x^k/\Gamma(\beta+k\stablebetawith)$ is the two-parameter Mittag-Leffler function, a complex-analytic function on all of $\bC$ for all $\beta>0$, and where $E_\stablebetawithout=E_{\stablebetawithout,1}$. 
\end{corollary}
\begin{proof} Since varying $a$ corresponds to a linear time change of $X$, we can assume $a=1$ w.l.o.g. By \cite{Ber96b}, $Z^{(q)}(\thresholda)=E_\stablebetawithout(q\thresholda^\stablebetawithout)$. Also
  $$\overline{\Pi}(\thresholda)=\int_\thresholda^\infty\frac{\stablebetawith\stablebetaminusonewith}{\Gamma(1-\stablebetaminusonespecial)}x^{-\stablebetaminusonespecial-2}dx=\frac{\stablebetaminusonewithout}{\Gamma(1-\stablebetaminusonespecial)}\thresholda^{-\stablebetaminusonespecial-1}.$$
With $\chi(dx)=\overline{\Pi}(x)dx$,  from \eqref{amicale},
\begin{eqnarray}\Theta^\thresholda(q)\!\!\!&=&\!\!\!\frac{1}{Z^{(q)}(\thresholda)}\left(\int_0^\thresholda(Z^{(q)}(\thresholda)-Z^{(q)}(\thresholda-x))\chi(dx)+(Z^{(q)}(\thresholda)-1)\chi((\thresholda,\infty))\right)\nonumber\\
                          &=&\!\!\!\frac{\stablebetaminusonewith\int_0^\thresholda(E_\stablebetawithout(q\thresholda^\stablebetawithout)-E_\stablebetawithout(q(\thresholda-x)^\stablebetawithout))x^{-\stablebetaminusonespecial-1}dx+\thresholda^{-\stablebetaminusonespecial}(E_\stablebetawithout(q\thresholda^\stablebetawithout)-1)}{Z^{(q)}(\thresholda)\Gamma(1-\stablebetaminusonespecial)}.\label{ratio}
                          \end{eqnarray}
We use power series techniques to calculate the first term of the numerator:
\begin{align*}&\int_0^\thresholda(E_\stablebetawithout(q\thresholda^\stablebetawithout)-E_\stablebetawithout(q(\thresholda-x)^\stablebetawithout))x^{-\stablebetaminusonespecial-1}dx\\
       &=\sum_{k=1}^\infty\frac{q^k}{\Gamma(1+k\stablebetawith)}\int_0^\thresholda(\thresholda^{k\stablebetawith}-(\thresholda-x)^{k\stablebetawith})x^{-\stablebetaminusonespecial-1}dx.
\end{align*}
We now solve this integral using Fubini theorem and Beta integrals to find finite coefficients
\begin{align*}&\frac{1}{\Gamma(1+k\stablebetawith)}\int_0^\thresholda\left(\thresholda^{k\stablebetawith}-(\thresholda-x)^{k\stablebetawith}\right)x^{-\stablebetaminusonespecial-1}dx\\
   &=\frac{\thresholda^{k\stablebetawith-\stablebetaminusonespecial}}{\Gamma(1+k\stablebetawith)}\frac{1}{\stablebetaminusonewithout}\left(\frac{\Gamma(1-\stablebetaminusonespecial)\Gamma(k\stablebetawith+1)}{\Gamma(k\stablebetawith+1-\stablebetaminusonespecial)}-1\right).
\end{align*}
The first term gives coefficients of the two-parameter Mittag-Leffler function. The second term gives coefficients just as needed to cancel with most of the second term in the numerator of (\ref{ratio}), we obtain the formula we claimed for $c=1$. To pass to general $c>0$, it is easy to check that $Z^{(q)}(\thresholda)=E_\stablebetawithout(q\thresholda^\stablebetawithout/c)$, so we can just replace $q$ by $q/c$. 
\end{proof}

\subsection{Uniform local time estimates up to random times}

\begin{proposition}\label{prop14} Consider a spectrally positive stable L\'evy process $X$ with Laplace exponent $\eta^{\stablebetawithout }$. There is $K>0$ such that for all $x,y\in\bR$ and all $N,r\in(0,\infty)$ we have
  $$\bP\left(\sup_{0\le t\le N}\left|\ell^y(t)-\ell^x(t)\right|>r\right)\le 2e^N\exp\left(-Kr|y-x|^{-\stablebetaminusonewith /2}\right).$$
  Hence for all $p>0$, we also have 
  $\displaystyle\bE\left(\sup_{0\le t\le N}\left|\ell^y(t)-\ell^x(t)\right|^p\right)\le 2\Gamma(p+1)K^{-p}e^N|b-a|^{p\stablebetaminusonewith /2}$.
\end{proposition}

\begin{lemma}\label{lm15} For $z>0$, let $\varphi(z)\in(0,1)$ be such that we have $(\varphi(z))^2=1-\Psi_0(z)\Psi_z(0)$, where $\Psi_x(y)=\bE_x(\exp(-T_y(X)))$ and  $T_y(X)=\inf\{t\ge 0\colon X_t=y\}$. Then there is $K>0$ such that $2\varphi(z)\le z^{\stablebetaminusonewith /2}/K$ for all $z>0$. 
\end{lemma}
\begin{proof} We may assume that $z\leq 1$ since the bound is trivially true for $z\geq 1$ provided $K\leq 1$.  Note that $\lim_{y\to 0} T_y(X) =0$ almost surely and, consequently, Lemma \ref{lm7} and the dominated convergence theorem combine to imply that for $\lambda\sim{\rm Exp}(1)$
\[ 1 = \lim_{y\downarrow 0} \bP(T_y(X) <\lambda) = \frac{\stablebetawithout }{\pi}\int_0^\infty\frac{\sin(\pi \stablebetaminusonewith ) s^{\stablebetawithout } }{s^{2\stablebetawith }+2s^{\stablebetawithout } \cos(\pi\stablebetaminusonewith )+1}ds.\]
Consequently,
\[  \begin{split}(\varphi(z))^2=1-\Psi_0(z)\Psi_z(0) = &\ 1 - \exp(-z)\frac{\stablebetawithout }{\pi}\int_0^\infty\frac{\sin(\pi \stablebetaminusonewith ) s^{\stablebetawithout } }{s^{2\stablebetawith }+2s^{\stablebetawithout } \cos(\pi\stablebetaminusonewith )+1} e^{-zs}ds \\
 = &\  \frac{\stablebetawithout }{\pi}\int_0^\infty\frac{\sin(\pi \stablebetaminusonewith ) s^{\stablebetawithout } }{s^{2\stablebetawith }+2s^{\stablebetawithout } \cos(\pi\stablebetaminusonewith )+1}(1- e^{-(s+1)z})ds \\
 \le & \ \frac{\stablebetawith \sin(\pi\stablebetaminusonewith )}{\pi} z \int_0^{z^{-1}}\frac{ s^{\stablebetawithout } }{s^{2\stablebetawith }+2s^{\stablebetawithout } \cos(\pi\stablebetaminusonewith )+1}(1+s)ds\\
 & \quad+ \frac{\stablebetawith \sin(\pi\stablebetaminusonewith )}{\pi} \int_{z^{-1}}^\infty\frac{ s^{\stablebetawithout } }{s^{2\stablebetawith }+2s^{\stablebetawithout } \cos(\pi\stablebetaminusonewith )+1}ds. 
    \end{split}\]
  Splitting $s^{2\stablebetawith}=s^{2\stablebetawith}\sin^2(\pi\stablebetaminusonewith)+s^{2\stablebetawith}\cos^2(\pi\stablebetaminusonewith)$, we can further bound 
  above by
  \begin{align*}
  &z+\frac{\stablebetawithout }{\pi \sin(\pi\stablebetaminusonewith )} z \int_0^{z^{-1}} s^{-\stablebetaminusonewith } ds+ \frac{\stablebetawith }{\pi \sin(\pi\stablebetaminusonewith )}\int_{z^{-1}}^\infty s^{-1-\stablebetaminusonespecial }ds\\
   &= z+\frac{\stablebetawithout }{(1-\stablebetaminusonespecial ) \pi \sin(\pi\stablebetaminusonewith )} z^{\stablebetaminusonewithout } + \frac{\stablebetawith }{\stablebetaminusonewith  \pi \sin(\pi\stablebetaminusonewith )} z^{\stablebetaminusonewithout } \\
   &\le z^\stablebetaminusonewithout  \left(1+\frac{\stablebetawithout }{(1-\stablebetaminusonespecial ) \pi \sin(\pi\stablebetaminusonewith )}+  \frac{\stablebetawith }{\stablebetaminusonewith  \pi \sin(\pi\stablebetaminusonewith )} \right),
  \end{align*}
and the result follows.
\end{proof}

\begin{proof}[Proof of Proposition \ref{prop14}] By \cite[Proposition V.(3.28)]{BlG68}, we have for all $N,\delta\in(0,\infty)$ and $x,y\in\bR$
  $$\bP\left(\sup_{0\le t\le N}\left|\ell^y(t)-\ell^x(t)\right|>2\delta\right)\le 2e^N\exp\left(-\delta/\varphi(|y-x|)\right),$$
  and setting $r=2\delta$, $z=|y-x|$ and with $K$ as in the preceding lemma, we conclude for the first inequality. The second is an elementary integration of the first. 
\end{proof}
  
In fact, the statement on moments, when expressed in terms of $\varphi$, also holds in the full generality of the Blumenthal-Getoor result, any ``standard Markov process''. Since the tail estimate is also independent of the starting value of $X$, by the strong Markov property at $T_x\wedge T_y$, it turns out that it is easy to improve the dependence on $N$ in the moment bounds, in the general case. In particular, we
obtain the following.   
  
\begin{corollary}\label{cor16} There is $K>0$ such that for all $p>0$, $x,y\in\bR$ and $N>0$
  $$\bE\left(\sup_{0\le t\le N}\left|\ell^y(t)-\ell^x(t)\right|^p\right)\le 2\Gamma(p+1)K^{-p}(N+1)^{p+1}e|y-x|^{p\stablebetaminusonewith /2}.$$
\end{corollary}
\begin{proof} We will use
  $\displaystyle\sup_{0\le t\le N}\left|\ell^y(t)-\ell^x(t)\right|\le\sum_{k=0}^{\lfloor N\rfloor}\sup_{0\le s\le 1}\left|\left(\ell^y(k+s)-\ell^y(k)\right)-\left(\ell^x(k+s)-\ell^x(k)\right)\right|$
  and estimate the sum above, as follows
  \begin{equation}\label{elemineq}\left(\sum_{k=0}^{\lfloor N\rfloor}x_k\right)^p\le\left((N+1)\max_{0\le k\le\lfloor N\rfloor}x_k\right)^p\le(N+1)^p\sum_{k=0}^{\lfloor N\rfloor}x_k^p.
  \end{equation}
  By the Markov property and since local times are additive functionals, all $\lfloor N\rfloor+1$ terms satisfy the same moment bound from Proposition \ref{prop14}, applied for unit time intervals, so we obtain the estimate as claimed.
\end{proof}

\begin{corollary}\label{cor17} For any random time $V$ with moments of all orders and any $p>0$, there is $M_p$ depending on the distribution of $V$ such that for all $x,y\in\bR$, we have 
  $$\bE\left(\sup_{0\le t\le V}\left|\ell^y(t)-\ell^x(t)\right|^p\right)\le M_p|y-x|^{p\stablebetaminusonewith /2}.$$
\end{corollary}
\begin{proof} We use the Cauchy-Schwarz inequality to obtain
  \begin{eqnarray*}\bE\left(\sup_{0\le t\le V}\left|\ell^y(t)-\ell^x(t)\right|^p\right)\!\!\!&\le&\!\!\!\sum_{k\ge 2}\bE\left(1_{\{k-2\le V<k-1\}}\sup_{0\le t\le k-1}\left|\ell^y(t)-\ell^x(t)\right|^p\right)\\
     &\le&\!\!\!\sum_{k\ge 2}\Big(\bP(k\!-\!2\!\le\! V\!<\!k\!-\!1)\Big)^{1/2}\left(\bE\left(\sup_{0\le t\le k-1}\left|\ell^y(t)-\ell^x(t)\right|^{2p}\right)\right)^{1/2}\\
     &\le&\!\!\!\left(\sum_{k\ge 2}k^{p+1/2}(\bP(V\ge k-2))^{1/2}\right)\sqrt{2e\Gamma(2p+1)}K^{-p}|y-x|^{p\stablebetaminusonewith /2}.
  \end{eqnarray*}
  Since finite $p^\prime$th moment of $V$ implies that there is $k_0(p^\prime)\ge 1$ such that $\bP(V\ge k-2)\le k^{-p^\prime}$ for all $k\ge k_0(p^\prime)$, we can choose any $p^\prime>2p+3$ to see that the series in our expression converges.
\end{proof}

Now consider the spectrally positive stable-$\stablebetawith $ process restricted to an interval $[0,\thresholda]$. Recall that this process is defined as $X^\thresholda_r=X_{\tau^{[0,\thresholda]}(r)}$, where
$$\tau^{[0,\thresholda]}(r)=\inf\left\{t\ge 0\colon R^{[0,\thresholda]}(t)>r\right\},\qquad\mbox{and }R^{[0,\thresholda]}(t)=\int_0^t1_{\{0\le X_s\le \thresholda\}}ds.$$
We can use the fact that $X^\thresholda$ is a Markov process with the same local time differences as $X$ in the interval $[0,\thresholda]$ to improve on this corollary and effectively allow times $V=\tau^{[0,\thresholda]}(Q)$, where only $Q$ is required to have moments of all orders, subject only to restricting $x,y\in[0,\thresholda]$. This is a
genuine improvement because $\bE(\tau^{[0,\thresholda]}(r))=\infty$ for all $r\ge 0$, since there is positive probability that $X$ leaves $(-\infty,\thresholda]$ during $[0,r]$ after which to return inside $[0,\thresholda]$ it takes $X$ a level passage time given by a stable-$(1/\stablebetawith)$ ladder time subordinator.

\begin{proposition}\label{prop18} For any $\thresholda>0$, any random time $Q$ with moments of all orders and any $p>0$, there is $M_{p,\thresholda}$ depending on the distribution of $Q$ such that for all $x,y\in[0,\thresholda]$, we have 
  $$\bE\left(\sup_{0\le t\le\tau^{[0,\thresholda]}(Q)}\left|\ell^y(t)-\ell^x(t)\right|^p\right)\le M_{p,\thresholda}|y-x|^{p \stablebetaminusonewith /2}.$$ 
\end{proposition}
\begin{proof} We repeat some of the previous arguments. First, \cite[Proposition V.(3.28)]{BlG68} applies to   
  $X^\thresholda$ with local times $\ell_\thresholda^x(r)=\ell^x(\tau^{[0,\thresholda]}(r))$, $0\le x\le \thresholda$, $r\ge 0$: for all $N>0$, $\delta>0$, $x,y,z\in[0,\thresholda]$,
   $$\bP_z\left(\sup_{0\le t\le N}\left|\ell^y_\thresholda(r)-\ell_\thresholda^x(r)\right|\ge 2\delta\right)\le 2e^N\exp(-\delta/2\varphi_\thresholda(x,y)),$$
   where 
   \begin{eqnarray*}(\varphi_\thresholda(x,y))^2&=&1-\bE_x(\exp(-T_y(X^\thresholda)))\bE_y(\exp(-T_x(X^\thresholda)))\\
                    &\le& 1-\bE_x(\exp(-T_y(X)))\bE_y(\exp(-T_x(X)))\ \ =\ \ (\varphi(|y-x|))^2.
   \end{eqnarray*}
   By Lemma \ref{lm15} and integration, we find $K>0$ such that for all $p>0$
   $$\bE_z\left(\sup_{0\le r\le N}\left|\ell^y_\thresholda(r)-\ell^x_\thresholda(r)\right|^p\right)\le 2\Gamma(p+1)K^{-p}e^N|y-x|^{p\stablebetaminusonewith /2}.$$
   Now the arguments of Corollaries \ref{cor16} and \ref{cor17} apply to give $M_{p,\thresholda}$ so that
   $$\bE\left(\sup_{0\le t\le \tau^{[0,\thresholda]}(Q)}\left|\ell^y(t)-\ell^x(t)\right|^p\right)
     =\bE\left(\sup_{0\le r\le Q}\left|\ell^y_\thresholda(t)-\ell^x_\thresholda(t)\right|^p\right)\le M_{p,\thresholda}|y-x|^{p\stablebetaminusonewith /2}.\vspace{-0.5cm}$$
\end{proof}

\begin{theorem}\label{thmuniloc} For any $\thresholda>0$, any random time $Q$ with moments of all orders and any $\localtimeholder\in(0,\stablebetaminusonewith /2)$, the random variable 
 $$ D^\thresholda_\localtimeholder(\tau^{[0,\thresholda]}(Q))=\sup_{0\le t\le \tau^{[0,\thresholda]}(Q),0\le x<y\le \thresholda}\frac{|\ell^y(t)-\ell^x(t)|}{|y-x|^\localtimeholder}\ <\ \infty\qquad\mbox{a.s.}
  $$
has moments of all orders.
\end{theorem} 
\begin{proof} Now consider the Banach space of bounded continuous functions from $[0,\infty)$ to $\bR$ equipped with the supremum norm $||\cdot||_\infty$ (rather than any localised version). Then the stopped local time processes $L_v=(\ell^{\thresholda v}(t\wedge \tau^{[0,\thresholda]}(Q)),t\ge 0)$, $v\in[0,1]$, are members of this Banach space and satisfy
$$\bE(||L_v-L_u||_\infty^p)\le M_{p,\thresholda} |v-u|^{p\stablebetaminusonewith /2}\qquad\mbox{for all $u,v\in[0,1]$,}
$$
by Proposition \ref{prop18}. We can apply the Revuz-Yor version of the Kolmogorov-Chentsov theorem \cite[Theorem I.(2.1)]{RY} with $\varepsilon=p\stablebetaminusonewith /2-1$, to find that
$$\bE((D_\localtimeholder^a(\tau^{[0,\thresholda]}(Q)))^p)=\frac{1}{M_{p,\thresholda}}\bE\left(\left(\sup_{0\le u<v\le 1}\frac{||L_v-L_u||_\infty}{|v-u|^\localtimeholder}\right)^p\right)<\infty,$$
as long as $\localtimeholder\in(0,\varepsilon/p)$, i.e. if $\localtimeholder<(\stablebetaminusonewith /2)-(1/p)$, which gives the result for $\localtimeholder\in (0,\stablebetaminusonewith /2)$ by letting $p\to\infty$. 
\end{proof}

Now consider the Poisson point process of excursions of $X$ away from 0, enriched by independent 
${\rm BESQ}(-2\stablebetaminusonewith )$ excursions in each jump as in Section \ref{secmarkbes}. Mark each excursion with probability $1-e^{-m}$, where $m$ is the value of the ${\rm BESQ}(-2\stablebetaminusonewith )$ excursion in the jump of $X$ across 0 when crossing 0. Denote by $S$ the point process local time of the first mark and by
$T$ the left endpoint of the corresponding excursion of $X$. 

\begin{corollary} For any $\thresholda>0$ and $\localtimeholder\in(0,\stablebetaminusonewith /2)$, the random variable
  $$D^{[0,\thresholda]}_\localtimeholder:=D^{[0,\thresholda]}_\localtimeholder(T)=\sup_{0\le t\le T,0\le x<y\le \thresholda}\frac{|\ell^y(t)-\ell^x(t)|}{|y-x|^\localtimeholder}\ <\ \infty\qquad\mbox{a.s.}
  $$
and has moments of all orders.
\end{corollary}
\begin{proof} First note that $T=\tau^{[0,\thresholda]}(\sigma^\thresholda_{S-})\le \tau^{[0,\thresholda]}(\sigma^\thresholda_S)$. By elementary properties of Poisson point processes, $S$ is exponentially distributed, with parameter $\mu$, say. Note, however, that $S$ is not independent of $\sigma^\thresholda$. We use the same argument as in Corollary \ref{cor17} and (\ref{elemineq}) to find
  $$\bE((\sigma_S^\thresholda)^p)\le\sum_{k\ge 1}\left(\bP(S\!\ge\! k\!-\!1)\right)^{1/2}\left(\bE\left(\!\left(\sigma_k^\thresholda\right)^{2p}\right)\!\right)^{1/2}\!\!\le\!\left(\sum_{k\ge 1}e^{-\mu k/2}k^{p+1/2}\!\right)e^{\mu/2}\left(\bE\!\left(\!\left(\sigma_1^\thresholda\right)^{2p}\right)\!\right)^{1/2}\!\!.$$
  By Proposition \ref{prop12}, the inverse local time $\sigma^\thresholda_1$ of $X^\thresholda$ has moments of all
  orders. Also the series clearly converges. Hence $\sigma_S^\thresholda$ has moments of all orders, as does $\sigma_{S-}^\thresholda<\sigma_S^a$. Hence, Theorem \ref{thmuniloc} applies to $Q=\sigma_{S-}^\thresholda$, and this completes the proof.
\end{proof}

\subsection{Proof of Theorem \ref{thmltholdermoment}}

\begin{proof}[Proof of Theorem \ref{thmltholdermoment}] We repeat the argument of the previous proof. Specifically, we apply Theorem \ref{thmuniloc} to 
  $Q=R^{[0,\thresholda]}(\tau^0(S))=\sigma^b_S$.
\end{proof}

\bibliographystyle{plain}
\bibliography{adabbrev2}

\noindent
\end{document}